\documentclass[12pt,a4paper]{article}
\usepackage{amsmath,amssymb,amsthm,graphicx,color}
\usepackage[left=1in,right=1in,top=1in,bottom=1in]{geometry}
\usepackage{tikz}
\usetikzlibrary{topaths,calc}
\usepackage{cite}
\usepackage[british]{babel}
\usepackage{enumitem}

\usepackage{hyperref} 
\hypersetup{
    colorlinks=false,
    pdfborder={0 0 0},
}

\newtheorem{theorem}{Theorem}[section]
\newtheorem{corollary}[theorem]{Corollary}
\newtheorem{proposition}[theorem]{Proposition}
\newtheorem{lemma}[theorem]{Lemma}
 \newtheorem{conjecture}[theorem]{Conjecture}

\theoremstyle{definition}

\theoremstyle{remark}
\newtheorem{remark}[theorem]{Remark}

\newtheorem*{remark*}{Remark}

 \allowdisplaybreaks

 \newcommand{\1}{{\mathbf 1}}

\newcommand{\R}{{\mathbb R}}
\newcommand{\Z}{{\mathbb Z}}

\newcommand{\N}{{\mathbb N}}
\newcommand{\ZP}{{\mathbb Z}_+}
\newcommand{\RP}{{\mathbb R}_+}
\newcommand{\Sp}[1]{{\mathbb S}^{#1}}

\DeclareMathOperator{\Exp}{\mathbb{E}}
\renewcommand{\Pr}{{\mathbb P}}

\DeclareMathOperator{\Int}{int}

\DeclareMathOperator{\sd}{\triangle}

\newcommand{\conv}{\mathop \mathrm{Conv}}
\newcommand{\vol}[1]{\mathop {\mathrm{Vol}_{#1}}}
\newcommand{\cl}{\mathop \mathrm{cl}}
\newcommand{\cone}{\mathop \mathrm{Cone}}

\newcommand{\unif}{\mathrm{Unif}}

\newcommand{\eps}{\varepsilon}

\newcommand{\nn}{\nonumber}

\newcommand{\re}{{\mathrm{e}}}
\newcommand{\rc}{{\mathrm{c}}}
\newcommand{\ud}{{\mathrm d}}

\newcommand{\calA}{{\mathcal A}}
\newcommand{\calF}{{\mathcal F}}
\newcommand{\calG}{{\mathcal G}}
\newcommand{\calH}{{\mathcal H}}
\newcommand{\calP}{{\mathcal P}}
\newcommand{\calX}{{\mathcal X}}

\newcommand{\fB}{{\mathfrak B}}

\newcommand{\as}{\ \text{a.s.}}

\newcommand{\tod}{\overset{\text{d}}{\longrightarrow}}
\newcommand{\eqd}{\overset{d}{=}}

\newcommand{\bigmid}{\; \bigl| \;}
\newcommand{\Bigmid}{\; \Bigl| \;}
\newcommand{\biggmid}{\; \biggl| \;}

\newcommand{\0}{{\mathbf{0}}}

\makeatletter
\def\namedlabel#1#2{\begingroup  
    (#2)%
    \def\@currentlabel{#2}%
    \phantomsection\label{#1}\endgroup
}
\makeatother

\begin{document}

\title{Random walks avoiding their convex hull\\ with a finite memory}
\author{Francis Comets\footnote{NYU Shanghai and Universit\'e Paris Diderot,  Math\'ematiques, case 7012, 75205 Paris Cedex 13, France; 
\texttt{\href{mailto:comets@lpsm.paris}{comets@lpsm.paris}}}
\and Mikhail V.\ Menshikov\footnote{Department of Mathematical Sciences, Durham University, South Road, Durham DH1 3LE, UK;
\texttt{$\{$\href{mailto:mikhail.menshikov@durham.ac.uk}{mikhail.menshikov},\href{mailto:andrew.wade@durham.ac.uk}{andrew.wade}$\}$@durham.ac.uk}}
 \and Andrew R.\ Wade\footnotemark[2]}

\date{\today}
\maketitle

\begin{abstract}
Fix integers $d \geq 2$ and $k\geq d-1$.
Consider a random walk $X_0, X_1, \ldots$ in $\R^d$ in which, given $X_0, X_1, \ldots, X_n$ ($n \geq k$),
the next step $X_{n+1}$ is uniformly distributed on the unit ball centred at $X_n$, but
conditioned that the line segment from $X_n$ to $X_{n+1}$ intersects the convex hull
of $\{0, X_{n-k}, \ldots, X_n\}$ only at $X_n$.
  For $k = \infty$ this is a version of the
model introduced by Angel \emph{et al.}, which is conjectured to be ballistic,
i.e., to have a limiting speed and a limiting direction. We establish
ballisticity for the finite-$k$ model, and comment on some open
problems. In the case where $d=2$ and $k=1$, we obtain the limiting speed explicitly: it is
$8/(9\pi^2)$.
\end{abstract}

\medskip

\noindent
{\em Key words:}  Random walk; convex hull; rancher; self-avoiding; ballistic; speed.

\medskip

\noindent
{\em AMS Subject Classification:}  60K35 (Primary) 60G50, 52A22, 60F15 (Secondary).

\section{Introduction and main results}

Random walks in Euclidean space whose evolution depends not just upon their most recent state  
but upon their previous history have  recently attracted much interest. 
Such non-Markov processes arise naturally in systems where there is learning,
resource depletion, or physical interaction. For example,
 a roaming animal performing a random walk may tend to avoid previously visited regions
to access new resources~\cite[\S 4]{smouse}. Another major motivation is to provide models in polymer science, where linear chain molecules naturally appear both in collapsed and extended phases~\cite{bn}.

A broad class of models is provided by random walks (or diffusions) that interact with the occupation measure of
their  
past  trajectory. This interaction can be local, such for reinforced~\cite{pemantle} or excited random walks~\cite{BW03}, in which the walker's
motion is biased by its occupation measure in the immediate vicinity,
or global, such as for processes with self-interaction mediated via some
global functional of the past trajectory, such as a centre of mass or other occupation statistic~\cite{BFV10,cmvw,mta,nrw,toth, toth2, tothwerner}.
In either case, the self-interaction can be attractive, corresponding to the collapsed polymer phase,
or repulsive, corresponding to the extended phase. 
An important distinction exists between dynamic models, that are genuine stochastic processes,
and static models, such as the self-avoiding walk~\cite{MS}, which is the extreme case in which repulsion is total.

Locally self-repelling walks in continuous space also appear in queueing theory, as models of customer-server systems with greedy strategies:
customers arrive randomly in time and space and the server moves toward the closest customer between services. Questions of interest include stability when the space is the circle~\cite{RSTournier}, and transience and rate of escape on the line~\cite{FRSidoravicius, RollaS}. Analogues in discrete space are considered in~\cite{CW,KM97},
where it is shown that in different regimes the server's trajectories mimic either the self-avoiding or correlated random walk~\cite{cr}.

The inspiration for the work in the present paper originates with
a model of Angel~\emph{et al.}~in which the random walk is forbidden from entering the
\emph{convex hull} of its previous trajectory~\cite{ABV03,Zerner}.
This model, known as the \emph{rancher process}, is believed to be \emph{ballistic} (see below), but no proof of this exists at the moment.
A lattice-based model which shares some common features with the rancher process
is the \emph{prudent random walk}~\cite{BFV10} in which the walker avoids travelling in a direction towards a previously visited site. 
It is worth noting that the scaling limits of the prudent walk in its kinetic version~\cite{BFV10} and its static (uniform) version~\cite{BM10, PST17} are quite different.
In this paper we consider a 
variant of the rancher model in continuous space for which we can establish ballisticity.

 Let us describe our model,  a version of the rancher problem 
which retains memory only of a fixed number of its recent locations together with its initial point (the origin). Fix $d \geq 2$ (the ambient dimension)
and 
$k \in \N$ (the length of the memory of the walk) with $k \geq d-1$ (this condition rules out degenerate cases, as we explain below).
Our object of interest is the stochastic process
$X = (X_0, X_1, X_2, \ldots)$ in $\R^d$ 
where, roughly speaking, given $X_0, \ldots, X_n$, the next position
$X_{n+1}$ is uniformly distributed on the unit ball centred at $X_n$
but conditioned so that the line segment from $X_n$ to $X_{n+1}$ does not intersect the convex hull of $\{ 0, X_{n-k}, X_{n-k+1}, \ldots, X_n \}$
at any point other than $X_n$ (which is necessarily on the boundary of the hull).

To give the formal definition, 
we write $\conv \calX$ for the convex hull of a subset $\calX \subseteq \R^d$, that is, the smallest
convex set containing $\calX$.
Let $B(x;r)$ denote the closed Euclidean $d$-ball
centred at $x \in \R^d$ with radius $r>0$, and
for $x, y \in \R^d$ let
$(x,y] := \{  {\lambda x + (1-\lambda) y} : \lambda \in (0,1] \}$, which for $x \neq y$ is the line segment from $x \in \R^d$ to $y \in \R^d$ excluding $x$.
The set of \emph{admissible states} from $x \in \R^d$ with history $\calX \subseteq \R^d$ is 
\begin{align}
\label{eq:admissible1}
\calA(\calX ; x ) & := \cl \big\{ y \in  B ( x ; 1): ( x, y]  \cap 
\conv ( \calX \cup \{ 0, x\} ) = \emptyset \big\} ;
\end{align}
 where taking the closure (`$\cl$') is convenient for some measurability statements.

Let $\vol{d}$ denote Lebesgue measure on $\R^d$,
and set
$\calX_{n,k} := \{ X_j : \max(1,n-k) \leq j \leq n-1 \}$.
We define the law of $X$ by taking $X_0 =0$
and declaring that, for $n \in \ZP$,
\begin{equation}
\label{eq:transition-rule}
 \Pr ( X_{n+1} \in A \mid X_0, X_1, \ldots, X_n ) = \int_A p ( y \mid \calX_{n,k} ; X_n ) \ud y, \end{equation}
for all Borel sets $A \subseteq \R^d$, where $p$ is the transition density
defined for $y \in \R^d$ by 
\begin{equation} 
\label{def:walkk}
p(y \mid \calX ; x ) = \frac{1}{\vol{d} \calA( \calX  ; x)} \1_{\calA( \calX  ; x)}(y)
\end{equation} 
if $\vol{d} \calA( \calX  ; x) > 0$; i.e.,
 given $X_0, \ldots, X_n$, the next step $X_{n+1}$ is uniform on $\calA( \calX_{n,k} ; X_n)$.
We call $X$ the \emph{random walk with memory $k$}. 
The definition is analogous to the ones in \cite{ABV03, Zerner} for the `infinite memory' case.
See Figure~\ref{fig1} for an illustration in $d=2$.

Lemma~\ref{lem:well-defined} below shows that we only need to define $p(y \mid \calX  ; x)$ when
$\vol{d} \calA( \calX; x ) > 0$; hence the process $X_0, X_1, \ldots$ is well defined.

Note that we do not allow $k \leq d-2$. Indeed, in that case
$\conv ( \calX_{n,k} \cup \{ 0, X_n\} )$ has at most $d$ vertices,
so it is contained in a $(d-1)$-dimensional hyperplane,
and $\calA ( \calX_{n,k} ; X_n)$ is, up to a set of measure zero,
the whole of $B(X_n;1)$: the random walk has no interaction with its history, and has independent jumps. 

\begin{figure}[!h]
\centering
\includegraphics[width=0.2\textwidth]{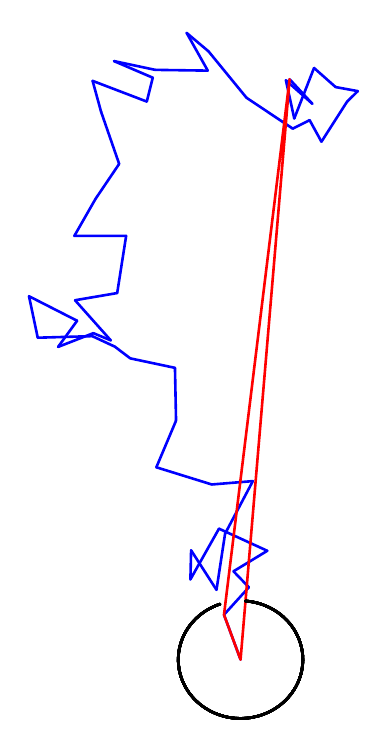}
\qquad\qquad
\includegraphics[width=0.33\textwidth]{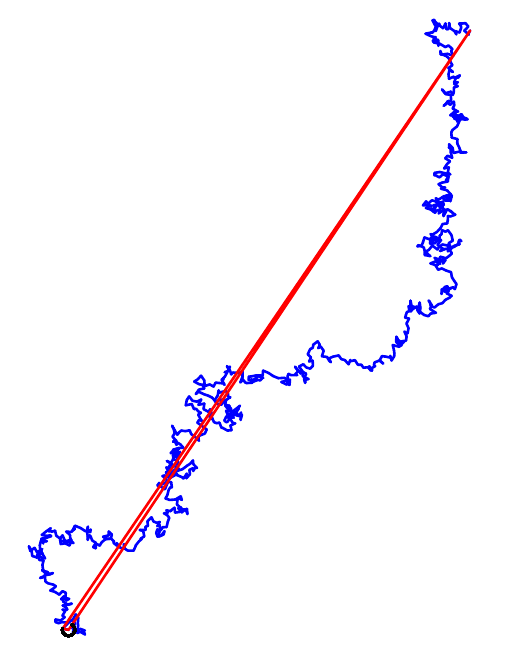}
\caption{Simulation of the $d=2$, $k=1$ process for $50$ steps (\emph{left})
and $1000$ steps (\emph{right}). The trajectory is in blue, the convex hull in red, and the black arc describes the disk sector on which the next position is distributed.}
\label{fig1}
\end{figure}

The main aim of this paper is to prove that this random walk is \emph{ballistic}, i.e.,
it has a positive asymptotic speed and a limiting direction. Here is the theorem.
Write $\| \, \cdot \, \|$ for the Euclidean norm on $\R^d$ and set
 $\Sp{d-1} := \{ u \in \R^d : \| u \| = 1 \}$.

\begin{theorem} 
\label{thm:speed}
There exist a positive
constant $v_{d,k}$ and a uniformly distributed $\ell \in \Sp{d-1}$ such that
\[   \lim_{n \to \infty} \frac{X_n}{n} = v_{d,k} {\ell}, \as, \text{ and hence } \lim_{n \to \infty} \frac{\Exp \| X_n\| }{n} = v_{d,k} .  \]
\end{theorem}

Note that including the origin in the definition of the convex hull to be
  avoided at each step is crucial;
if the process instead just avoids the convex hull generated by its most recent $k$ steps, then it will be diffusive,
like the Gillis--Domb--Fisher `correlated random walk'~\cite{cr} that is repelled by its immediate past but effectively has zero drift
over long time scales. Another model whose dynamics, like ours, are influenced by both
its very distant and very recent past was considered recently by Gut and Stadtm\"uller~\cite{gs} and is
a variant of the `elephant random walk'~\cite{bb,bercu}. For their model on $\Z$, Gut and Stadtm\"uller
obtain a ballisticty result reminiscent of Theorem~\ref{thm:speed}: see Theorem~10.1 of~\cite{gs}.

The constants $v_{d,k}$ in Theorem~\ref{thm:speed} are characterized in~\eqref{eq:v} below, but seem hard to evaluate in general. 
It is obvious that $v_{d,k} \leq 1$,
and we show (cf.~Corollary~\ref{cor:exponential-bound}) that $v_{d,k} >0$. It is likely that
 one can show that $v_{d,k} \geq v >0$ for all $d,k$, perhaps by adapting the arguments of~\cite{Zerner};
this fact would also follow from Conjecture~\ref{conj1} below.
We can compute $v_{d,k}$ explicitly in one particular case.

\begin{theorem}
\label{thm:memory-one}
If $d=2$ and $k=1$, then
\[
 v_{2,1} = \frac{8}{9\pi^2} \approx 0.09006327 . \]
\end{theorem}

Simulations suggest the following.

\begin{conjecture}
\label{conj1}
We have $v_{d,k} \leq v_{d,k+1}$ for all $k \geq d-1$.
\end{conjecture}

It is natural to seek a coupling to establish Conjecture~\ref{conj1}. There is an obvious coupling
of one step of the $k$ and $k+1$ processes started from a common configuration, but extending this to a process coupling seems difficult. 

 The inspiration for considering our model comes
from the case of \emph{infinite memory},
when the walk avoids its entire convex hull
$\conv \{ X_{0}, X_{1},\ldots, X_n \}$. This  `$k=\infty$' walk is a variant of the model introduced by Angel \emph{et al.}\ in~\cite{ABV03},
in which the increments are uniform on the unit \emph{sphere}
(rather than the unit ball)
excluding the convex hull;
for the $d=2$ case of that model, Zerner~\cite{Zerner} showed that $\liminf_{n \to \infty} n^{-1} \| X_n \| >0$ a.s.
Just as for the model in~\cite{ABV03}, one conjectures that the $k=\infty$ walk that avoids its entire convex hull is ballistic
(cf.~Conjectures~1 and~5 in \cite{ABV03});
in particular, one expects that $\lim_{n \to \infty} n^{-1} \| X_n \| = v_{d,\infty}$ exists.
Our Conjecture~\ref{conj1} would imply that $\lim_{k \to \infty} v_{d,k}$ exists. It is then tempting to propose the following.
 
\begin{conjecture}
\label{conj2}
We have $\lim_{k \to \infty} v_{d,k} = v_{d,\infty}$.
\end{conjecture}

Simulations  are reasonably consistent with Conjecture~\ref{conj2}, but not entirely convincing.
Another open problem concerns the second-order behaviour of $X_n$ in the finite-memory model:
we expect that $n^{-1/2} ( X_n - v_{d,k} n \hat X_n )$ converges to a non-degenerate normal
distribution; this is to be contrasted with the conjectured $n^{3/4}$-order fluctuations (in $d=2$) for the $k=\infty$ model~\cite{ABV03}. It is also open to prove ballisticity for the version of the finite-memory model
($X_n'$, say)
in which the increments are supported on a sphere rather than a ball: 
our proof (particularly the renewal construction in Section~\ref{sec:renewal}) uses the fact that
the increments have a density in $\R^d$. In the case $d=2$, $k=1$ of this variant of the model, 
the argument
of Section~\ref{sec:memory-1} goes through with minor modifications to show that
\[ \lim_{n\to\infty} n^{-1} \Exp \| X'_n \| = \frac{4}{3\pi^2} \approx 0.13509491.\]

The plan of the paper is as follows. In Section~\ref{sec:prelim}
we collect some initial observations, which include a
description of the process via a $(k+1)$-component Markov chain
and the fact that there is a uniformly positive radial drift for the process
over a finite number of steps, which entails a $\liminf$-speed bound.
The core of our proof of ballisticity is a renewal structure described in Section~\ref{sec:renewal},
which identifies events that occur frequently and between any two of which the process
has uniformly positive radial drift and has symmetric transverse increments. This is essentially already
enough to prove a limiting direction, but to identify a limiting speed it is necessary to show that 
the radial drift between renewals has a limit, and that the expected time between renewals also has a limit.
We establish these limiting statements via a coupling argument to a 
 variant
of the process which is spatially homogeneous. The homogeneous process is introduced in Section~\ref{sec:homogeneous},
and the coupling argument is presented in Section~\ref{sec:coupling}. This completes the proof of Theorem~\ref{thm:speed}.
The proof of Theorem~\ref{thm:memory-one} proceeds via an essentially self-contained argument in Section~\ref{sec:memory-1}, which shows that $n^{-1} \Exp \| X_n \|$ has the specified limit.
The argument goes by showing that the global speed is asymptotically equal to the local drift, and  the local drift is evaluated as 
an average with respect to the limit distribution of the interior angle of the convex hull;
the limit distribution of the angle is identified in 
Lemma~\ref{lem:theta-convergence} as the limit of 
the stochastic recursive sequence 
\[ \theta_{n+1}= \big\vert (2 \pi-\theta_n) U_{n+1}-\pi \big\vert \]
 taking values in $[0,\pi]$, where $U_1, U_2, \ldots$ are i.i.d.~uniform.
The technical results required to deduce limiting speed and direction from statements about increments are collected in the Appendix.

\section{Preliminaries}
\label{sec:prelim}

For any finite non-empty $\calX \subseteq \R^d$ and 
any $x \in \R^d$, let
\[ \cone(x;\calX) := \conv \{ x + \alpha (y-x ) : \alpha \geq 0, \, y \in \calX \} . \]
Excluding the degenerate case $\cone( x; \{ x \} ) = \{ x \}$, 
$\cone(x;\calX)$ is the convex hull of finitely many closed rays emanating from $x$,
and, if $x$ is not in the interior of $\conv \calX$, then $\cone(x;\calX)$ is the smallest closed convex cone with vertex $x$ containing the set $\calX$
(equivalently, $\conv \calX$). It is not hard to see that~\eqref{eq:admissible1} is equivalent to
\begin{equation}
\label{eq:admissible2}
\calA(\calX ; x ) = \cl \bigl( B(x ; 1) \setminus\cone (x ; \calX \cup \{0,x\}) \bigr),
\end{equation}
which is a form that will be useful later on.

Our first result of this section shows that our process is well defined.
Here and subsequently, $\nu_d := \vol{d} B(0;1)$ is the volume of the unit-radius
$d$-ball.

\begin{lemma}
\label{lem:well-defined}
The process $X_0, X_1, X_2, \ldots$ is well defined, and for all $n \in \ZP$, a.s.,
\begin{align}
\label{eq:A-bounds} 
   \frac{\nu_d}{2} \leq 
	\vol{d} \calA ( \calX_{n,k} ; X_n)  \leq \nu_d .
\end{align}
\end{lemma}
\begin{proof}
The proof goes by induction. 
Starting from  $X_0=0$
 we have that $\calA  ( \calX_{0,k} ; X_0 ) = \calA  ( \emptyset ; 0 ) = B (0 ; 1)$ by~\eqref{eq:admissible1} or~\eqref{eq:admissible2}.
Hence $\vol{d} \calA  ( \emptyset ; 0 ) = \nu_d$,
so~\eqref{eq:A-bounds} holds with $n=0$.
For the inductive step, suppose that 
the law of $X_0, X_1, \ldots, X_{m}$
is well defined, and that~\eqref{eq:A-bounds} 
holds for all $0 \leq n \leq m$.
Then the transition density $p(y \mid \calX_{n,k} ; X_n)$
at~\eqref{def:walkk} is well-defined for $n=m$, and so we can generate $X_{m+1}$ according to~\eqref{eq:transition-rule}.
Thus $X_0, X_1, \ldots, X_{m+1}$ is well defined: see Figure~\ref{fig:half-disc} for an example.
Moreover, the upper bound on $\vol{d} \calA  ( \calX_{m+1,k} ; X_{m+1})$
is trivial. Also, by construction, $X_{m+1}$  is not in the interior of the previous convex hull $\conv (\calX_{m,k} \cup \{ 0, X_{m} \} )$, and so $X_{m+1}$ is extremal 
for $\conv ( \calX_{m+1,k} \cup \{ 0, X_{m+1} \})$. Hence there exists a tangent hyperplane at $X_{m+1}$ to the convex hull, and the opposite half of the ball $B(X_{m+1};1)$ is contained in  
$\calA ( \calX_{m+1,k} ; X_{m+1} )$. 
Hence the latter set has volume at least $\nu_d/2$, and so~\eqref{eq:A-bounds}
holds for $n=m+1$. This completes the inductive step.
\end{proof}

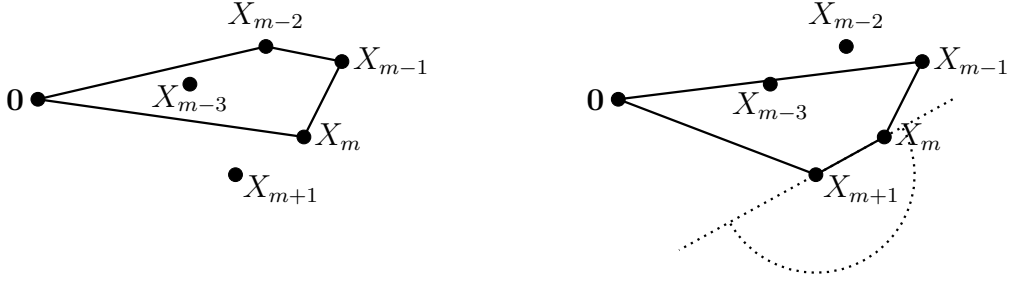
\begin{figure}[!h]
\centering
\begin{tikzpicture}
\draw [color=white,dotted,line width=0.3mm,domain=210:390] plot ({(2.6)+(1.3)*cos(\x)}, {(-1)+(1.3)*sin(\x)});
\draw[fill=black] (0,0) circle (.5ex); 
\draw[fill=black] (3,0.7) circle (.5ex); 
\draw[fill=black] (4,0.5) circle (.5ex); 
\draw[fill=black] (3.5,-0.5) circle (.5ex); 
\draw[fill=black] (2.6,-1) circle (.5ex); 
\draw[fill=black] (2,0.2) circle (.5ex); 
\draw[line width=0.3mm] (0.0,0.0) -- (3.0,0.7); 
\draw[line width=0.3mm] (3.0,0.7) -- (4.0,0.5); 
\draw[line width=0.3mm] (4.0,0.5) -- (3.5,-0.5); 
\draw[line width=0.3mm] (3.5,-0.5) -- (0,0); 
\node at (-0.3,0.0) {$\0$};
\node at (4.65,0.5) {$X_{m-1}$};
\node at (3,1.1) {$X_{m-2}$};
\node at (3.95,-0.5) {$X_{m}$};
\node at (3.2,-1.2) {$X_{m+1}$};
\node at (2.0,0) {$X_{m-3}$};
\end{tikzpicture}
\qquad\qquad
\begin{tikzpicture}
\draw[fill=black] (0,0) circle (.5ex); 
\draw[fill=black] (3,0.7) circle (.5ex); 
\draw[fill=black] (4,0.5) circle (.5ex); 
\draw[fill=black] (3.5,-0.5) circle (.5ex); 
\draw[fill=black] (2.6,-1) circle (.5ex); 
\draw[fill=black] (2,0.2) circle (.5ex); 
\draw[line width=0.3mm] (0.0,0.0) -- (4.0,0.5); 
\draw[line width=0.3mm] (4.0,0.5) -- (3.5,-0.5); 
\draw[line width=0.3mm] (3.5,-0.5) -- (2.6,-1);
\draw[line width=0.3mm] (0,0) -- (2.6,-1);
\draw[dotted,line width=0.3mm] (4.4,0) -- (0.8,-2);
\node at (-0.3,0.0) {$\0$};
\node at (4.65,0.5) {$X_{m-1}$};
\node at (3,1.1) {$X_{m-2}$};
\node at (3.95,-0.5) {$X_{m}$};
\node at (3.2,-1.2) {$X_{m+1}$};
\node at (2.0,-0.1) {$X_{m-3}$};
\draw [dotted,line width=0.3mm,domain=210:390] plot ({(2.6)+(1.3)*cos(\x)}, {(-1)+(1.3)*sin(\x)});
\end{tikzpicture}
\caption{An illustration of the proof of Lemma~\ref{lem:well-defined} with $d = k=2$. \emph{Left:} the new point $X_{m+1}$
sits outside
$\conv ( \calX_{m,k} \cup \{ 0, X_m \} )$.
\emph{Right:} 
Indicated is half of the disc $B(X_{m+1} ; 1)$ which falls outside the updated convex hull $\conv ( \calX_{m+1,k} \cup \{ 0, X_{m+1} \} )$, giving a lower bound
on the area of $\calA (\calX_{m+1,k};X_{m+1})$.
}
\label{fig:half-disc}
\end{figure}
 
For $m \in \ZP$ define
\[ Y_{m} := (X_{m k},X_{m k +1},\ldots,X_{(m+1)k} ). \]
Let $x_0, \ldots, x_k \in \R^d$.
A sequence $(y_1,\ldots ,y_k)$ of points in $\R^d$ is called an \emph{admissible path} from history 
$x_0, x_1, \ldots, x_k$ if 
$y_1  \in \calA( x_0, \ldots, x_{k-1} ; x_k )$, 
$y_2 \in \calA( x_1, \ldots, x_{k-1} , x_k ; y_1 )$,
$y_3 \in \calA (x_2,\ldots, x_k, y_1; y_2)$,
and so on, up to
$y_k \in \calA (x_{k}, y_1 \ldots, y_{k-2} ; y_{k-1} )$.
Let $\calP ( x_0, x_1, \ldots, x_k )$ denote the set
of all admissible paths from history $x_0, x_1, \ldots, x_k$.
To describe the initial steps of the process,
we say  $y_1,\ldots ,y_k \in \R^d$ is  an  admissible initial path
if $y_1 \in \calA ( \emptyset ; 0 )$, $y_2 \in \calA ( \emptyset ; y_1)$,
$y_3 \in \calA (y_1 ; y_2)$, and so on, up to
$y_k \in \calA (y_1, \ldots, y_{k-2} ; y_{k-1})$.
 Let $\calP_0$ denote the set
of all admissible initial paths.

The $\calP ( x_0, x_1, \ldots, x_k )$ are Borel (in fact, closed) subsets of $(\R^d)^k$.
To see this, 
take $(y_{1,n}, \ldots, y_{k,n}) \in \calP( x_0, x_1, \ldots, x_k )$
with  $(y_{1,n}, \ldots, y_{k,n}) \to (y_1,\ldots,y_k)$ as $n \to \infty$.
Since $\calA( x_0, \ldots, x_{k-1} ; x_k )$ is closed,  $y_1 \in \calA( x_0, \ldots, x_{k-1} ; x_k )$.
Moreover, as a function taking values in the non-empty compact subsets of $\R^d$ endowed with the Hausdorff metric (denoted $\rho_H$),
  $y_1 \mapsto \calA( x_1, \ldots, x_{k-1} , x_k ; y_1 )$ is continuous, and so
$\lim_{n \to \infty} \calA( x_1, \ldots, x_{k-1} , x_k ; y_{1,n} ) = \calA( x_1, \ldots, x_{k-1} , x_k ; y_1 )$.
Given $\eps >0$, we can (and do) choose $n$ sufficiently large so that $\| y_{2,n} - y_2 \| < \eps$
and $\rho_H ( \calA( x_1, \ldots, x_{k-1} , x_k ; y_{1,n} ) ,  \calA( x_1, \ldots, x_{k-1} , x_k ; y_1 ) ) < \eps$.
Then since $y_{2,n} \in \calA( x_1, \ldots, x_{k-1} , x_k ; y_{1,n} )$, there exists $z_n \in \calA( x_1, \ldots, x_{k-1} , x_k ; y_1 )$
with $\| y_{2,n} - z_n \| < \eps$, so that $\| z_n - y_2 \| < 2\eps$. Hence $y_2 = \lim_{n\to \infty} z_n \in \calA( x_1, \ldots, x_{k-1} , x_k ; y_1 )$,
since the latter set is closed. Continuing this argument shows that $(y_1,\ldots,y_k) \in \calP ( x_0, x_1, \ldots, x_k )$,
so the latter set is closed.
Similarly, $\calP_0$ is a closed subset of  $(\R^d)^k$.

\begin{lemma}
\label{lem:Y-Markov}
The process $Y = (Y_0, Y_1, Y_2, \ldots)$ is a Markov process
on $(\R^{d})^{k+1}$ with transition function defined
for all Borel sets $\fB \subseteq  (\R^{d})^{k}$  by
\[ \Pr ( Y_{m+1} \in \{ x_k \} \times \fB \mid Y_m = (x_0, \ldots, x_k) ) = 
\int_{\fB \cap \calP ( x_0, \ldots, x_k)} \!\!\! f ( y_1,  \ldots, y_k \mid x_0, \ldots, x_k ) \ud y_1 \cdots \ud y_k ,
\]
where for all $(y_1, \ldots, y_k ) \in \calP ( x_0, \ldots, x_k)$,
with $p ( \, \cdot \mid   \cdot\, )$ given by~\eqref{def:walkk},
\begin{align}
\label{eq:f-def}
  f ( y_1, \ldots, y_k \mid x_0, x_1, \ldots, x_k )  
& = p ( y_1 \mid x_0, x_1, \ldots, x_{k-1} ; x_k ) \nn\\
& {}\quad{} \times p (y_2 \mid x_1, \ldots, x_{k} ; y_1 ) \nn\\
& {}\quad{} \times p ( y_3 \mid x_2, \ldots, x_k, y_1 ; y_2 ) \nn\\
& {}\quad{} \times \cdots \times p ( y_k \mid x_k, y_0, \ldots , y_{k-2} ; y_{k-1} ),
\end{align}
and elsewhere we set $f = 0$.
Moreover, the initial distribution is
\[ \Pr ( Y_0 \in \{ 0 \} \times \fB ) = \int_{\fB \cap \calP_0} f_0( y_1,  \ldots, y_k ) \ud y_1 \cdots \ud y_k ,\]
where for all $(y_1, \ldots, y_k ) \in \calP_0$,
\begin{align*}
  f_0 ( y_1, \ldots, y_k  )  
& = p ( y_1 \mid \emptyset ; 0 )   p (y_2 \mid \emptyset ; y_1 ) 
 p ( y_3 \mid  y_1 ; y_2 )  \cdots  p ( y_k \mid  y_1, \ldots , y_{k-2} ; y_{k-1} ).
\end{align*}
\end{lemma}
\begin{proof}
It suffices to suppose that $\fB = \prod_{i=1}^k \fB_i$ for $\fB_i$ Borel sets in $\R^d$. By~\eqref{eq:transition-rule} and~\eqref{def:walkk},
\[ \Pr ( X_{(m+1)k+1} \in \fB_1 \mid Y_m = (x_0, \ldots, x_k) )
= \int_{\fB_1 \cap \calA ( x_0 , \ldots, x_{k-1} ; x_k )} p ( y_1 \mid x_0, x_1, \ldots, x_{k-1} ; x_k )  \ud y_1,  \]
which gives the result if $k=1$. Otherwise,
\begin{align*} 
& {} \qquad {} \Pr ( ( X_{(m+1)k+1} , X_{(m+1)k+2} ) \in \fB_1 \times \fB_2 \mid Y_m = (x_0, \ldots, x_k) ) \\
& = \int_{\fB_1 \cap \calA ( x_0 , \ldots, x_{k-1} ; x_k )} \Pr ( X_{(m+1)k+2} \in \fB_2 \mid X_{mk+1} = x_1, \ldots, X_{(m+1)k} = x_k, X_{(m+1)k +1 } = y_1 )  \\
& {} \qquad\qquad\qquad\qquad\qquad\qquad\qquad {} \times p ( y_1 \mid x_0, x_1, \ldots, x_{k-1} ; x_k )\ud y_1  \\
& = \int_{\fB_1 \cap \calA (x_0, \ldots, x_{k-1} ; x_k )}
\int_{\fB_2  \cap \calA (x_1, \ldots, x_{k} ; y_1 )} p ( y_2 \mid x_1, \ldots, x_k ; y_1 ) p ( y_1 \mid x_0,  \ldots, x_{k-1} ; x_k ) \ud y_2 \ud y_1  
,\end{align*}
which gives the result if $k=2$. Iterating this argument gives the transition function for general $k$.
A similar argument gives the law of $Y_0$.
\end{proof}
 
For $n \in \ZP$ define the $\sigma$-algebra $\calF_n := \sigma ( X_0, X_1, \ldots, X_n)$.
For $x \in \R^d \setminus \{ 0\}$, define $\hat x := x / \| x \|$.
For convenience,   set $\hat 0 := 0$. We write `$\, \cdot \, $' for the scalar product on $\R^d$.
The following important result says that the radial component of the drift of the process
is always non-negative.

\begin{proposition}
\label{prop:centrifugal}
We have that, for all $n \in \ZP$,
\begin{align*}
 \Exp [ ( X_{n+1} - X_n ) \cdot \hat X_n \mid \calF_n ] & \geq 0, \as
\end{align*}
\end{proposition}
\begin{proof}
Given that $\hat 0 =0$, it suffices to suppose that $n \geq 1$, in which case $X_n \neq 0$, a.s.  On the event $\{X_n=x\}$, for $x \neq 0$,
by~\eqref{eq:transition-rule} and~\eqref{def:walkk} we can write
\begin{align*}
  \Exp [ ( X_{n+1} - X_n ) \cdot \hat X_n \mid \calF_n ]  &= \frac{1}{\vol{d} (\calA (\calX_{n,k};x))} 
 \int_{ \calA (\calX_{n,k};x) } (y-x) \cdot \hat x  \ud y \\
& =  \frac{1}{\vol{d} (\calA (\calX_{n,k};x))} 
 \int_{ \calA' (\calX_{n,k};x) } (y-x) \cdot \hat x  \ud y,
\end{align*}
where the open set
\begin{equation}
\label{eq:admissible3}
 \calA' (\calX ; x ) = \Int \left(  B(x ; 1) \setminus\cone (x ; \calX \cup \{0,x\}) \right) , \end{equation}
differs from $\calA (\calX ; x)$ as given by~\eqref{eq:admissible2}
by  a set of measure zero (`$\Int$' stands for `interior').
For $x \neq 0$, define $S_x:  z \mapsto z-2\big((z-x)\cdot \hat x\big) \hat x$, the orthogonal transformation of $\R^d$
induced by reflection in the hyperplane at $x$ orthogonal to $\hat x$. We claim that
\begin{equation} \label{eq:symm}
y \in \calA'(\calX_{n,k};x) \text{ and } (y-x) \cdot \hat x < 0 \text{ imply that } S_x(y) \in \calA'(\calX_{n,k};x).
\end{equation}
Write $\calA^+ = \{ y \in \calA'(\calX_{n,k};x) : (y-x) \cdot \hat x > 0 \}$
and $\calA^- = \{ y \in \calA'(\calX_{n,k};x) : (y-x) \cdot \hat x < 0 \}$. Then by~\eqref{eq:symm} and the fact that
$(S_x (z) - x ) \cdot \hat x = - (z-x) \cdot \hat x$, we have
$S_x (\calA^-) \subseteq \calA^+$ and, using also the fact that $S_x$ is a measure-preserving bijection,
\begin{equation}
\label{eq:flip}
 \int_{S_x (\calA^-)}  (y-x) \cdot \hat x  \ud y =  \int_{\calA^-} ( S_x (z) - x) \cdot \hat x \ud z = - \int_{ \calA^- } (z-x) \cdot \hat x  \ud z .\end{equation}
Hence, partitioning $\calA^+$ into $S_x (\calA^-)$ and $\calA^+ \setminus S_x (\calA^-)$, we get
\begin{align}
\label{eq:radial-mean}
 \int_{ \calA' (\calX_{n,k};x) } (y-x) \cdot \hat x  \ud y & =  \int_{ \calA^- } (y-x) \cdot \hat x  \ud y
+ \int_{S_x (\calA^-)}  (y-x) \cdot \hat x  \ud y \nonumber\\ & {} \qquad {}
+ \int_{\calA^+ \setminus S_x (\calA^-)}  (y-x) \cdot \hat x  \ud y \nonumber\\
& = \int_{\calA^+ \setminus S_x (\calA^-)}  (y-x) \cdot \hat x  \ud y,
\end{align}
using~\eqref{eq:flip}. Moreover, the integrand in the final integral in~\eqref{eq:radial-mean}
is positive, by definition of $\calA^+$. Thus we conclude that the final integral in~\eqref{eq:radial-mean}
is non-negative.

It remains to prove the claim~\eqref{eq:symm}. To do so, we use a finite-dimensional version of Hahn--Banach theorem: 
for all $y \in  \calA'(\calX_{n,k};x)$, there exists a hyperplane $H$ separating  $\{y\}$ and
$\cone (x;\calX_{n,k} \cup \{0,x\})$ such that $y \notin H$ and
\[
H \bigcap \cone (x;\calX_{n,k} \cup \{0,x\}) = \{x\}  ;
\]
here it is important that we used $\calA'$ defined at~\eqref{eq:admissible3}. 
Consider the unit vector $h$ perpendicular to $H$ and such that $h \cdot \hat x >0$, and denote by $H^+, H^-$ the half-spaces 
\[
H^\pm = \{ z \in  \R^d : \pm (z-x) \cdot h >0 \}.
\]
Then, for $y \in \calA'(\calX_{n,k};x) $ and $(y-x) \cdot \hat x < 0 $, we have
$y \in H^+$ and $S_x(y) \in H^+$ though $\cone (\calX_{n,k} \cup \{0,x\}) \subseteq \cl H^-$. Since also $S_x(y) \in \Int B(x;1)$,
we have from~\eqref{eq:admissible3} that $S_x(y) \in \calA'(\calX_{n,k};x)$. Thus
we have proved~\eqref{eq:symm}.
\end{proof}

We would like to improve Proposition~\ref{prop:centrifugal} to show that the radial drift 
is uniformly positive. However, it is not hard to see that there are configurations
for which this is not true if we compute the drift in a single step. Thus we are led to consider multiple steps.
In order to control the possible configurations of the walk's history, we can demand that the walk first
makes a chain of jumps away from the convex hull, and then makes another chain of jumps in the radial direction.
These two constructions will be central to our renewal structure that we describe in the next section,
and they are the focus of the next two results.

For $x \in \R^d$, $\delta \in (0,1/8)$, and any unit vector $u \in \Sp{d-1}$, 
define 
\begin{equation}
\label{eq:pi-def}
 \Pi^{u} (x) := \prod_{i=1}^k B  \bigl( x +  \tfrac {i}{2} u ; \delta \bigr) \subseteq (\R^d)^{k}.
\end{equation}
Given $\calF_n$, consider a tangent hyperplane at $X_n$ to   $\conv(\calX_{n,k} \cup \{ 0, X_n \})$, and 
let $h$ be the perpendicular unit vector to this hyperplane, pointing opposite to 
the convex hull.
We show that from any configuration, the walk will follow the chain laid out by $\Pi^h (X_n)$
with uniformly positive probability.

\begin{lemma} 
\label{lem:fact0}
We have that
\[  \Pr ( (X_{n+1} , \ldots, X_{n+k} ) \in \Pi^h (X_n) \mid \calF_n ) \geq \delta^{dk} , \as \]
\end{lemma}
\begin{proof}
Suppose that $X_n = x$; note that $x \cdot h \geq 0$.
Let $x_i = x + \frac{i}{2}h$ for $0 \leq i \leq k$.
Define the events $A_i = \{ X_{n+i} \in B ( x_i ; \delta )\}$.
It is easy to see that $B ( x_1 ; \delta) \subseteq \calA (\calX_{n,k} ; X_n)$.
Hence, by~\eqref{eq:transition-rule} and~\eqref{def:walkk},
\begin{equation}
\label{eq:case1}
 \Pr ( A_1 \mid \calF_n ) = \frac{\vol{d} B ( x_1 ; \delta )}{\vol{d} \calA (\calX_{n,k} ; X_n)} \geq \delta^d ,\end{equation}
by Lemma~\ref{lem:well-defined}. If $k=1$, this completes the proof. In general, we claim that
\begin{equation}
\label{eq:casei}
 \Pr ( A_{i+1} \mid \calF_{n+i} ) \geq \delta^d , \text{ on } \cap_{j=1}^i A_i .\end{equation}
Then, for instance,
\[ \Pr ( A_1 \cap A_2 \mid \calF_n ) = \Exp [ \Pr ( A_2 \mid \calF_{n+1} ) \1 ( A_1 ) \mid \calF_n ] \geq \delta^{2d} , \as,\]
by~\eqref{eq:case1} and~\eqref{eq:casei}. Iterating this argument proves the statement in the lemma.

It remains to prove the claim~\eqref{eq:casei}.
For $1 \leq i \leq k$, consider the hyperplane $H_i = \{ y \in \R^d : (y-x_i) \cdot h = 0 \}$.
Any $z \in \conv(\calX_{n,k} \cup \{ 0, X_n \})$ has $(z-x_i) \cdot h \leq -i/2 < -\delta$, and,
for $j < i$,  $(x_j - x_i ) \cdot h \leq -1/2 < -\delta$, while $(x_{i+1} -x_i) \cdot h = 1/2 > \delta$. Hence $x_i$ lies on $H_i$, and the hyperplane separates
$\conv(\calX_{n,k} \cup \{ 0, X_n \})$ and all the $B ( x_j ; \delta )$, $j < i$, from $B ( x_{i+1} ; \delta )$.
Thus, $B (x_{i+1} ; \delta ) \subseteq \calA ( \calX_{n+i,k} ; X_{n+i} )$ on $\cap_{j=1}^i A_i$.
The claim~\eqref{eq:casei} follows.
\end{proof}

With the notation at~\eqref{eq:pi-def}, define
$\Pi (x) := \Pi^{\hat x} (x) =  \prod_{i=1}^k B  ( x +  \tfrac {i}{2} \hat x ; \delta )$. 
The key to our renewal structure is the following definition: 
\begin{equation} \label{eq:G}
 \calG := \left\{ (x_0,\ldots,x_k) \in (\R^d)^{k+1} : x_k \neq 0, \text{ and } \Pi (x_k) \subseteq \calP ( x_0, \ldots, x_k) \right\}.
\end{equation}
For $n \geq k$, let $G_n \in \calF_n$ denote the event $G_n := \{ (X_{n-k},\ldots, X_n) \in  \calG \}$;
if $G_n$ occurs, we say that $X$ has \emph{good geometry} at time $n$.

 Roughly speaking,
the process has good geometry if the configuration is such that, in the next $k$ steps,
all trajectories through the 
sequence of balls laid out by $\Pi$ are admissible.
More precisely, the next result shows that, if the process has good geometry, then 
the law of the next $k$ steps has a uniform component on 
the balls laid out by $\Pi$.

\begin{lemma}
\label{lem:fact1} 
For all $n \geq k$ and all Borel $\fB \subseteq (\R^d)^k$, on the event $G_n$,
\[
\Pr (  (X_{n + 1},\ldots, X_{n+k}) \in \fB  \mid \calF_n ) \geq \delta^{dk} \frac{  \vol{dk} \fB}{ ( \vol{d} B(0;\delta) )^k} \1 \{ \fB \subseteq \Pi ( X_n) \} .
\]
\end{lemma}
\begin{proof}
For $n \geq k$ suppose that $\fB \subseteq \Pi ( X_n)$,
where $\fB = \prod_{i=1}^k \fB_i$ for $\fB_i$ Borel subsets of $\R^d$.
On the event $G_n$, we have that $\Pi (X_n) \subseteq \calP ( X_{n-k}, \ldots, X_n)$. 
In particular, $\fB_1 \subseteq B ( X_n + \frac{1}{2} \hat X_n ; \delta ) \subseteq \calA ( \calX_{n,k} ; X_n )$,
so that, by~\eqref{eq:transition-rule} and~\eqref{def:walkk} we have,
on $G_n \cap \{ \fB \subseteq \Pi ( X_n) \}$,
\begin{align*}
\Pr (  X_{n + 1} \in \fB_1 \mid \calF_n ) 
& =  \frac{ \vol{d} \fB_1}{ \vol{d} \calA ( \calX_{n,k} ; X_n)}  \geq  \frac{ \vol{d} \fB_1 }{ \vol{d} B(0;1)} ,\end{align*}
by Lemma~\ref{lem:well-defined}. Hence
\[ \Pr (  X_{n + 1} \in \fB_1 \mid \calF_n ) \geq \delta^d \frac{ \vol{d} \fB_1}{ \vol{d}   B(0 ; \delta) } 
 . \] 
If $k=1$ this ends the proof. Otherwise, on $G_n \cap \{ \fB \subseteq \Pi ( X_n) \} \cap \{ X_{n+1} \in \fB_1 \}$,
we have that $\fB_2 \subseteq B ( X_n + \frac{2}{2} \hat X_n ; \delta ) \subseteq  \calA ( \calX_{n+1,k} ; X_{n+1} )$,
so that
\begin{align*}
\Pr (  X_{n + 2} \in \fB_2 \mid \calF_{n+1} ) 
& =  \frac{ \vol{d} \fB_2}{ \vol{d} \calA ( \calX_{n+1,k}; X_{n+1})}  \geq  \delta^d \frac{ \vol{d} \fB_2}{ \vol{d}   B(0 ; \delta) } ,\end{align*}
as before. Hence, on $G_n \cap \{ \fB \subseteq \Pi ( X_n) \}$,
\begin{align*} \Pr ( (X_{n+1}, X_{n+2} ) \in \fB_1 \times \fB_2 \mid \calF_n )
& \geq \Exp [ \Pr (  X_{n + 2} \in \fB_2 \mid \calF_{n+1} ) \1 \{ X_{n+1} \in \fB_1 \} \mid \calF_n ]\\
& \geq  \delta^{2d}  \frac{ \vol{2d} ( \fB_1 \times \fB_2) }{ (\vol{d}   B(0 ; \delta) )^2} .\end{align*}
Iterating this argument gives the result.
\end{proof}

Recall the definition of $h$ from just before Lemma~\ref{lem:fact0}. 
The connection between the last two lemmas is the following.

\begin{lemma}
\label{lem:sufficient-for-gg}
We have that $(X_{n+1} , \ldots, X_{n+k} ) \in \Pi^h (X_n)$ implies $G_{n+k}$.
\end{lemma}
\begin{proof}
Suppose that $(X_{n+1} , \ldots, X_{n+k} ) \in \Pi^h (X_n)$. Let $x = X_{n+k}$.
Let $(y_1, \ldots, y_k) \in \Pi (x)$. We must show that $(y_1, \ldots, y_k ) \in \calP ( X_n ,\ldots, X_{n+k} )$.
For convenience, set 
$x_i = x + \frac{i}{2} \hat x$ for $0 \leq i \leq k$, and set
$z_i = X_{n+k-i}$ for $1 \leq i \leq k$.

It is not hard to see that  $y_1 \in \calA ( z_k, \ldots, z_1 ; x )$.
We have $\| y_i - x_i \| \leq \delta$ for $1 \leq i \leq k$,
and $\| z_i - z_i'\| \leq 2\delta$ where $z'_i = x - \frac{i}{2} h$.
For $1 \leq i \leq k-1$, consider the hyperplane $H_i = \{ y \in \R^d : (y-y_i) \cdot \hat x = 0 \}$.
Then $(x-x_i ) \cdot \hat x = - \frac{i}{2}$ so $(x-y_i) \cdot \hat x < 0$.
For all $j$ we have $(z'_j - x_i ) \cdot \hat x \leq - \frac{i}{2} \leq - 4\delta$,
so $(z_j - y_i ) \cdot \hat x < 0$.
Also, for $j < i$ we have $(x_j - y_i ) \cdot \hat x < 0$,
while $(x_{i+1} - y_i) \cdot \hat x > 0$.
Thus $H_i$ contains $y_i$ and separates
$x , z_1, \ldots, z_k$ and any $y_j$, $j < i$, from $y_{i+1}$.
In particular,   for $i=1$, this
 shows that $y_2 \in \calA ( z_{k-1}, \ldots, z_1 , x ; y_1 )$, and so on.
\end{proof}

Now we can state our result on positive radial drift over a number of steps.

\begin{proposition}
\label{prop:speed}
Suppose that $d \geq 2$ and $k \geq d-1$.
Then there exists a constant $c_{d,k}$ with $0 < c_{d,k} \leq 2k+d+1$ such that, for all $n \in \ZP$,
\begin{align}
\label{eq:norm-drift}
 \Exp [ \| X_{n+2k+d+1}\| - \| X_n \| \mid \calF_n ] & \geq c_{d,k}, \as 
\end{align}
\end{proposition}
\begin{proof}
We will show that there exist
constants $a, p >0$ (depending on $d$ and $k$) and an event $A \in \calF_{n+2k+d}$, 
such that  
\begin{align}
\label{eq:proj-drift-prob}
& {} \Pr ( A \mid \calF_n ) \geq p, \as, \text{ and } \\
\label{eq:proj-drift}
& {} \Exp [ ( X_{n+2k+d+1} - X_{n+2k+d} ) \cdot \hat X_{n+2k+d} \mid \calF_{n+2k+d} ]  \geq a, \text{ on } A .\end{align}
Note that for all $x , \Delta \in \R^d$,
  $\| x + \Delta \| \geq ( x + \Delta ) \cdot \hat x$,
so  $\| x + \Delta \| - \| x \| \geq \Delta \cdot \hat x$.
Hence, by Proposition~\ref{prop:centrifugal}, 
\begin{align*}
& {} \qquad {} \Exp [\| X_{n+2k+d+1} \| -  \| X_n \| \mid \calF_n ] \\
& = \Exp \bigg[ \sum_{i=0}^{2k+d} \Exp [\| X_{n+i+1} \| -  \| X_{n+i} \| \mid \calF_{n+i} ] \biggmid \calF_n \bigg]\\
& \geq 
  \Exp \bigg[ \sum_{i=0}^{2k+d}  \Exp [ (  X_{n+i+1} -  X_{n+i} ) \cdot \hat X_{n+i} \mid \calF_{n+i} ] \biggmid \calF_n \bigg] \\
&  \geq  \Exp \bigl[  \Exp [ ( X_{n+2k+d+1} - X_{n+2k+d} ) \cdot \hat X_{n+2k+d} \mid \calF_{n+2k+d} ] \1 (A) \bigmid \calF_n \bigr] ,
\end{align*}
which is bounded below by $ ap$, by the claims~\eqref{eq:proj-drift-prob} and~\eqref{eq:proj-drift}. This gives~\eqref{eq:norm-drift} with $c_{d,k} = ap$.
The rest of the proof establishes~\eqref{eq:proj-drift-prob} and~\eqref{eq:proj-drift}.

We describe the event $A$, which will comprise three successive events. 
Given $\calF_n$, let $h$ be the perpendicular unit vector to a tangent hyperplane at $X_n$ to   $\conv(\calX_{n,k} \cup \{ 0, X_n \})$,
pointing opposite to 
the convex hull.
Define the events
$A_1 = \{ (X_{n+1} , \ldots, X_{n+k} ) \in \Pi^h (X_n)\}$
and $A_2 = \{ (X_{n+k+ 1},\ldots, X_{n+2k}) \in \Pi (X_{n+k} ) \}$.
Then by Lemmas~\ref{lem:fact0}, \ref{lem:fact1}, and~\ref{lem:sufficient-for-gg},
we have that $\Pr (A_1 \cap A_2 \mid \calF_n ) \geq \delta^{2kd}$, a.s.

Suppose that $A_1 \cap A_2$ occurs and consider the situation at time $n+2k$.
Let $x_i = X_{n+k+i}$ for $0 \leq i \leq k$,
and let $x'_i = x_0 + \frac{i}{2} \hat x_0$.
On $A_2$, we have $\| x_i - x'_i \| \leq \delta$. 
Set $e_1 = \hat x_0$, 
 and let $\{ e_1, e_2, \ldots, e_d \}$
be an orthonormal basis for $\R^d$ containing $e_1$.

Next set 
$f_i = \sum_{j=1}^i e_j$ and let 
$y_i = x_k + \frac{1}{2} f_i$ for $0 \leq i \leq d$.
The idea is that, with positive probability, the process
will follow close to the path $y_0, y_1, \ldots, y_{d}$,
at which point it will have strictly positive drift after producing a convex hull which
contains, approximately, a simplex.
Set $z_i = X_{n+2k+i}$ for $0 \leq i \leq d$.
Define the events
\begin{align*} 
E_i := \{ X_{n+2k+i} \in B ( y_i ; \delta ) \} ~(1 \leq i \leq d), \text{ and }
A_3 := \cap_{i=1}^{d} E_i
 .\end{align*}
Then, on $E_i$, $\| z_i - y_i \| \leq \delta$. Suppose that $\delta>0$
is small enough so that $8 \delta \sqrt{d} < 1$.

Define hyperplanes $H_i = \{ y \in \R^d : ( y - z_i ) \cdot \hat f_{i+1} = 0 \}$.
First note that $z_0 = x_k$ and, for $0 \leq j < k$,
$(x_j - x_k ) \cdot e_1 \leq \frac{j-k}{2} + 2 \delta < - \delta$,
while $(y_1 - x_k) \cdot e_1 = \frac{1}{2} > \delta$,
so the hyperplane $H_0$ contains $x_k$ and separates
$0, x_0, x_1, \ldots, x_{k-1}$ from $B ( y_1 ; \delta)$. So, on $A_1 \cap A_2$,
we have $\Pr ( E_1 \mid \calF_n ) \geq \delta^d$. Now suppose that $1 \leq i \leq d-1$
and that $\cap_{j=1}^i E_j$ occurs.
For $0 \leq j \leq k$,
given $A_2 \cap E_i$, we have, noting that $f_i \cdot f_j = i \wedge j$ and $\| f_i \| \leq \sqrt{d}$,
\begin{align*}
 ( x_j - z_i ) \cdot \hat f_{i+1} & \leq (x_j' - y_i ) \cdot \hat f_{i+1} + 2 \delta \\
& \leq (x'_j - x'_k ) \cdot \hat f_{i+1} - \frac{1}{2} f_i \cdot \hat f_{i+1} + 3 \delta \\
& \leq \frac{j-k}{2} e_i \cdot \hat f_{i+1} - \frac{i}{2\sqrt{d}} + 3\delta \\
& \leq - \delta,
\end{align*}
provided that $8 \delta \sqrt{d} < 1$.
Similarly, for $1 \leq j < i$, given $E_1 \cap \cdots \cap E_i$,
\[ (z_j - z_i) \cdot \hat f_{i+1} \leq ( y_j - y_i) \cdot \hat f_{i+1} + 2\delta \leq -\frac{1}{2\sqrt{d}} + 2 \delta \leq - \delta .\]
On the other hand,
\[ (y_{i+1} - z_i ) \cdot \hat f_{i+1} \geq \frac{1}{2} e_{i+1} \cdot \hat f_{i+1} - \delta > \delta .\]
Thus the hyperplane $H_i$ contains $z_i$ and separates
$0, x_0, x_1, \ldots, x_{k}$ and $z_1, \ldots, z_{i-1}$ from $B ( y_{i+1} ; \delta)$.
Hence, on $A_1 \cap A_2 \cap ( \cap_{j=1}^i E_i )$, we have that $\Pr ( E_{i+1} \mid \calF_{n+2k+i} ) \geq \delta^d$.
Setting $A :=  A_1 \cap A_2 \cap A_3$, it follows that 
$\Pr (A \mid \calF_n ) \geq \delta^{2kd+d^2} =: p$ as required for~\eqref{eq:proj-drift-prob}.

It remains to prove~\eqref{eq:proj-drift}, i.e., to show that on $A$ there is a uniformly positive radial drift.
As above, let $x'_i = x_0 + \frac{i}{2} \hat x_0$ where $x_0 = X_{n+k}$.
Also let $y'_i = x'_k + \frac{1}{2} f_i$.
Define the simplex $\Delta'$ to be the convex polytope with
vertices $x'_k, y'_1, \ldots, y'_{d}$.
Define the barycentre of the vertices $w := \frac{1}{d+1}  ( x'_k + \sum_{i=1}^d y'_i ) 
= x'_k + \sum_{i=1}^d \frac{d-i+1}{2d+2} e_i$. 
Note that $y'_d = ( \| x_0 \| + \frac{k}{2} ) e_1  + \frac{1}{2} \sum_{j=1}^d e_j$, and so
\[ (w - y'_{d} ) \cdot y'_{d} 
= -\sum_{i=1}^d \frac{i}{2d+2} e_i \cdot y'_d 
= - \frac{1}{2d+2} \left( \| x_0 \| + \frac{k}{2} + \frac{d(d+1)}{4} \right).\]
Since $\| y'_d \| \leq \| x_0 \| + \frac{k}{2} + \frac{d}{2} \leq \| x_0 \| + \frac{k}{2} + \frac{d(d+1)}{4}$, it follows that
\[ (w - y'_d ) \cdot \hat y'_d \leq - \frac{1}{2d+2} .\] 
Let $\Delta$ denote the `approximate simplex' with vertices $X_{n+2k}, \ldots, X_{n+2k+d}$.
Then on $A$ we have that $\| X_{n+2k} - x'_k \| \leq \delta$ while $\| X_{n+2k+i} - y'_i\| \leq 2\delta$ for $1 \leq i \leq d$.
Now $(w - y'_d) \cdot \hat y'_d$ is continuous as a function of $x'_k$ and $y'_1, \ldots, y'_d$ away from $y'_d =0$,
so in particular we can choose $\delta >0$ small enough so that
\begin{equation}
\label{eq:ball-in-simplex}
 ( z - X_{n+2k+d} ) \cdot \hat X_{n+2k+d} \leq - \delta \text{ for all } z \in B ( w' ; \delta ) ,\end{equation}
where $w'$ is the barycentre of the vertices of $\Delta$.

 We claim that for $\delta$ small enough, $B (w'; \delta )$ is in the interior of $\Delta$.
Indeed, $w'$ is in the interior of $\Delta$ unless it degenerates to
a polytope of lower dimension. But $\vol{d} \Delta$ is a continuous function of
its vertices, and the volume is strictly positive when $\delta =0$ (since then $\Delta$ is a genuine simplex), so we can find $\delta >0$
small enough so that the claim holds. 
Hence, for small enough $\delta$,
$B (w'; \delta ) \subseteq \conv ( X_{n+2k}, \ldots, X_{n+2k+d} )$
and hence $B (w'; \delta ) \subseteq \conv ( \calX_{n+2k+d,k} \cup \{ 0 , X_{n+2k+d} \} )$.

Setting $x = X_{n+2k+d}$ and using analogous notation to the proof of
 Proposition~\ref{prop:centrifugal}, we have 
that $S_x ( B (w' ; \delta ) ) \subseteq \calA^+ \setminus S_x (\calA^- )$. Hence from~\eqref{eq:radial-mean}, on $A$,
\[ \Exp [ ( X_{n+2k+d+1} - X_{n+2k+d} ) \cdot \hat x \mid \calF_{n+2k+d} ]
\geq \frac{1}{\nu_d} \int_{S_x ( B ( w' ; \delta ) )} ( y-x) \cdot \hat x \ud y \geq \frac{1}{\nu_d} \delta^{d+1} ,\]
by~\eqref{eq:ball-in-simplex}. This gives~\eqref{eq:proj-drift} with $a=\delta^{d+1}/\nu_d$,
 and completes the proof.
\end{proof}

Having established a strictly positive radial drift, we can deduce that the
process has a positive `$\liminf$' speed. This is the next result.

\begin{corollary}
\label{cor:exponential-bound}
Suppose that $d \geq 2$ and $k \geq d-1$.
There exist constants $\rho := \rho_{d,k} >0$ and $n_{d,k} \in \N$, depending only on $d$ and $k$, such that
\begin{equation}
\label{eq:exponential-bound} \Pr ( \| X_n \| \leq  \rho n  ) \leq \re^{-\rho n} , \text{ for all } n \geq n_{d,k} .\end{equation}
Moreover, 
$\liminf_{n \to \infty} n^{-1} \| X_n \| \geq \rho$, a.s. 
\end{corollary}
\begin{proof}
Define the process $Z_m =\| X_{m(2k+d+1)} \| - c_{d,k} m$. Then, since $d \leq k+1$,
\[ | Z_{m+1} - Z_m | \leq | \| X_{(m+1)(2k+d+1)} \| - \| X_{m(2k+d+1)} \|  | +   c_{d,k} \leq 4k+2d +2 \leq 10 k, \as \]
Also, writing $\calF'_m = \calF_{m(2k+d+1)}$, we have
\[ \Exp [ Z_{m+1} - Z_m \mid \calF'_{m} ] = \Exp [  \| X_{(m+1)(2k+d+1)} \| - \|  X_{m(2k+d+1)} \|  \mid \calF_{m(2k+d+1)} ] -   c_{d,k} \geq 0, \as ,\]
by~\eqref{eq:norm-drift}. Hence $Z_m$ is a submartingale
with uniformly bounded increments, and we can apply the one-sided
Azuma--Hoeffding inequality (see Theorem~2.4.14 in~\cite{bluebook})
to obtain
\[ \Pr \Bigl( Z_{m} - Z_0 \leq -\frac{c_{d,k}}{2} m \Bigmid \calF'_0 \Bigr) \leq \exp \left( - \frac{c_{d,k}^2 m^2}{800 m k^2} \right) .\]
Hence, since $Z_0 = \| X_0\| = 0$, for some $\rho >0$ depending on $d$ and $k$, for all $m \in \ZP$,
\begin{equation}
\label{eq:exponential-bound-subseq}
 \Pr \Bigl( \| X_{m(2k+d+1)} \| \leq  \frac{c_{d,k}}{2}  m  \Bigr) \leq \re^{-\rho  m} . \end{equation}
Let $m = \left\lfloor \frac{n}{2k+d+1} \right\rfloor$.  
Since $n-2k-d-1 \leq m (2k+d+1) \leq n$, it follows that  
\begin{align*} \Pr \left( \| X_n \| \leq \frac{c_{d,k}}{11k} n \right) & = 
\Pr \left( \| X_{m (2k+d+1)} \| \leq \| X_n \| + 2k + d +1, \,  \| X_n \| \leq \frac{c_{d,k}}{11k} n \right) \\
& \leq \Pr \left( \| X_{m (2k+d+1)} \| \leq \frac{c_{d,k}}{2} m \right) ,
\end{align*}
for all $n$ sufficiently large. Then~\eqref{eq:exponential-bound-subseq} yields~\eqref{eq:exponential-bound}.
Finally, it follows from~\eqref{eq:exponential-bound} and the Borel--Cantelli lemma
that $\liminf_{n \to \infty} n^{-1} \| X_n \| \geq \rho$, a.s.
\end{proof}

\section{Renewal structure}
\label{sec:renewal}

Our strategy for establishing ballisticity is to show that, up to smaller order terms, 
there is a limiting positive radial drift and  the transverse fluctuations are not too big (cf.~Lemma~\ref{lem:ballistic} below).
As in Proposition~\ref{prop:speed},
it is clear that this property cannot be the case at every step of the walk. Our strategy is to find an embedded process
which has these properties at random times. We call these random times `renewals'. They are such that process executes a chain
of approximately radial jumps (cf.~Lemma~\ref{lem:fact1}). Such times  exhibit a symmetry which entails a positive radial drift,
and these times occur rather frequently, as we show in Lemma~\ref{lem:good-geometry} below. With Corollary~\ref{cor:exponential-bound}, 
this is already essentially enough to
establish a limiting direction. To establish a limiting speed, it is required in addition that the radial drift at these renewal
times, and the expected time between renewals, have limits; these quantities are not constant because the special r\^ole played by the origin means the process lacks homogeneity.
We address this with a coupling to a homogeneous modification of the process, which, roughly speaking, sends the origin away to infinity,
 as described in Section~\ref{sec:homogeneous}.

From this point on, we fix the constant $\delta \in (0,1/8)$.
Recall the definition of $\Pi (x) = \Pi^{\hat x} (x)$ from~\eqref{eq:pi-def}
and of $f$ from~\eqref{eq:f-def}. First we state
a consequence of Lemma~\ref{lem:fact1}.

\begin{corollary}
\label{cor:f-bound}
Let $n \geq k$. Set $\alpha := \delta^{dk}$.
For all $(y_1, \ldots, y_k )$, on the event $G_n$,
\[
  f ( y_1, \ldots, y_k \mid X_{n-k}, \ldots, X_n ) \geq   \frac{ \alpha \1_{\Pi ( X_n)} (y_1,\ldots,y_k ) }{ ( \vol{d} B(0;\delta) )^k}  . 
\]
\end{corollary}
\begin{proof}
Similarly to Lemma~\ref{lem:Y-Markov}, on the event $G_n \cap \{ \fB \subseteq \Pi (X_n ) \}$,
\[ \Pr ( (X_{n+1},\ldots, X_{n+k} ) \in \fB \mid X_{n-k}, \ldots, X_n ) = \int_\fB f (y_1 , \ldots, y_k \mid X_{n-k}, \ldots, X_n ) \ud y_1 \cdots \ud y_k .\]
Combined with Lemma~\ref{lem:fact1}, this means that, on $G_n$,
\[\int_\fB f (y_1 , \ldots, y_k \mid X_{n-k}, \ldots, X_n ) \ud y_1 \cdots \ud y_k \geq \alpha 
\frac{  \vol{dk} \fB}{ ( \vol{d} B(0;\delta) )^k} \1 \{ \fB \subseteq \Pi ( X_n) \} ,\]
which gives the result.
\end{proof}

Now we construct a version of $Y$ which exhibits the required renewal structure,
by introducing an additional source of randomness
via a sequence 
$V_1, V_2, \ldots$ of i.i.d.~Bernoulli random variables
with $\Pr ( V_i = 1 ) = \alpha = 1 -\Pr (V_i = 0)$, where
$\alpha \in (0,1)$ is the constant in Corollary~\ref{cor:f-bound}.
For now we call this new process $Y' = ( Y'_0, Y'_1, \ldots)$
with $Y'_m \in (\R^d)^{k+1}$;
we will soon show that $Y'$ has the same law as $Y$.
The process $Y'$ will be adapted to the filtration $\calF'_0, \calF'_1,\ldots$
defined by $\calF'_m := \sigma ( Y'_0, V_1, Y'_1, \ldots, V_m, Y'_m )$.
At the same time as constructing the process, we generate a sequence of \emph{renewal times}
as we shall describe. 
Roughly speaking, $m$ is a renewal time if $X$ has good geometry at time $mk$ and $V_{m+1} = 1$; it allows a construction of the process such that its future evolution after a renewal  
depends only on the current location $X_{mk}$, and not on the past.
Define the event $G'_{m} := \{ Y'_{m} \in  \calG \}$.

 To start with, we take $Y'_0$ to be distributed exactly as $Y_0$, as described in Lemma~\ref{lem:Y-Markov}.
  Given $Y'_0, Y'_1, \ldots, Y'_m$,
	suppose also that we have generated renewal times $\tau_1  < \tau_2 <\ldots < \tau_{J(m)}$.
Then we generate $Y'_{m+1}$ as follows.
\begin{enumerate}[leftmargin=*]
\item If $G'_{m}$ does not occur, 
then generate $Y'_{m+1}$ from $Y'_m$ using the transition function described in Lemma~\ref{lem:Y-Markov}.
\item If $G'_{m}$ does occur, then do the following.
 \begin{enumerate}[leftmargin=*]
\item If $V_{m+1}=1$, then  declare that $\tau_{J(m)+1} = m$ is the next renewal time, and 
set $Y'_{m+1} = (Y'_{m,k+1}, Z_{m+1} )$, where $Y'_{m,i}$ is the $i$th component of $Y'_m$
and the vector $Z_{m+1} \in (\R^d)^k$ is uniformly distributed on $\Pi ( Y'_{m,k+1} )$.
\item If $V_{m+1}=0$, then 
set $Y'_{m+1} = (Y'_{m,k+1}, Z_{m+1} )$, where now $Z_{m+1}$ is generated according to the density $\hat f ( y_1 , \ldots, y_k \mid Y'_m )$ on  $\calP(Y'_m)$
 given by
\begin{equation}
\label{eq:f-hat}  
\hat f ( y_1 , \ldots, y_k \mid Y'_m ) :=  \frac{1}{1 - \alpha} \! \left[ 
f ( y_1 ,\ldots, y_k \mid Y'_m ) 
 - \frac{\alpha \1_{\Pi ( Y'_{m,k+1} )}( y_1, \ldots, y_k )}{( \vol{d} B(0;\delta) )^k}  
\right] ,
\end{equation}
where $f ( \, \cdot \mid   \cdot\, )$ is defined at~\eqref{eq:f-def} and
$\alpha \in (0,1)$ is the constant in Corollary~\ref{cor:f-bound}.
\end{enumerate}
\end{enumerate}
By Corollary~\ref{cor:f-bound},  $\hat f$
as defined at~\eqref{eq:f-hat} 
is non-negative on $G'_{m}$, and
since, by Lemma~\ref{lem:Y-Markov},
\[ \int_{\calP(Y_m)} f (y_1, \ldots, y_k \mid Y_m ) \ud y_1 \cdots \ud y_k = \Pr ( Y_{m+1} \in \{ Y_{m,k+1} \} \times \calP(Y_m) ) = 1 ,\]
where $Y_{m,i}$ is the $i$th component of $Y_m$,
we have that, on $G'_{m}$,
\[ \int_{\calP(Y'_m)}
\hat f ( y_1 , \ldots, y_k \mid Y'_m ) \ud y_1 \cdots \ud y_k
= 1 ,\]
so $\hat f$ is indeed a probability kernel.

\begin{lemma}
\label{lem:YY}
The process $Y'$ has the same law as the process $Y$ described at Lemma~\ref{lem:Y-Markov}.
\end{lemma}
\begin{proof}
By construction, $Y'_0$ has the same law as $Y_0$. Also by construction, we have that for Borel $\fB \subseteq (\R^d)^k$,
\begin{align}
\label{eq:YY}
  & \Pr ( Y'_{m+1} \in \{ x_k \} \times \fB \mid Y'_m = (x_0, \ldots, x_k) ) \nn\\
	& {} \qquad\qquad\qquad {} = 
\int_{\fB \cap \calP ( x_0, \ldots, x_k)} f ( y_1,  \ldots, y_k \mid x_0, \ldots, x_k ) \ud y_1 \cdots \ud y_k ,
\end{align}
on the complement of $G'_{m}$. It remains to show that~\eqref{eq:YY} also holds on $G'_{m}$.
For Borel $\fB \subseteq (\R^d)^k$, we have
\begin{align*}
& {} \qquad {} \Pr ( Y'_{m+1} \in \{ Y'_{m,k+1} \} \times \fB   \mid \calF'_m )\\
& = \Pr ( Y'_{m+1} \in \{ Y'_{m,k+1} \} \times \fB , \, V_{m+1} = 1 \mid \calF'_m ) +  \Pr ( Y'_{m+1} \in \{ Y'_{m,k+1} \} \times \fB , \, V_{m+1} = 0 \mid \calF'_m )\\
& = \alpha \frac{\vol{dk} ( \fB \cap \Pi ( Y'_{m,k+1} ) )}{\vol{dk}  \Pi ( Y'_{m,k+1} )}   + (1-\alpha) \int_{\fB \cap \calP ( Y'_m )} \hat f ( y_1,\ldots, y_k \mid Y'_m ) \ud y_1 \cdots \ud y_k \\
& = \int_{\fB \cap \calP ( Y'_m )}  f ( y_1,\ldots, y_k \mid Y'_m ) \ud y_1 \cdots \ud y_k ,
\end{align*}
by equation~\eqref{eq:f-hat}. This completes the proof.
\end{proof}

Since from $Y$ we can recover $X$, in view of Lemma~\ref{lem:YY}, we will from now on work on an enlarged probability space
and assume that the process  $Y$ (and hence $X$) is constructed as $Y'$, along with its renewal times.
We finish this section by showing that the renewal times must occur rather frequently.

\begin{lemma} 
\label{lem:good-geometry}
With $\alpha \in (0,1)$ the constant appearing in Corollary~\ref{cor:f-bound}, we have 
\begin{equation}
\label{eq:good-geometry-bound}
 \Pr ( G_{n+k} \mid \calF_n ) \geq \alpha, \as, \text{ for all } n \in \ZP. \end{equation}
Moreover, with $c>0$  given by $\re^{-c}=1-\alpha^2$, we have
\begin{equation}
\label{eq:tau-exponential}
\Pr ( \tau_{n+1} - \tau_n \geq 2 r \mid \calF'_{\tau_n+1} ) \leq  \re^{-c r} , \as, \text{ for all } r \geq 0 \text{ and all } n \in \N.\end{equation}
\end{lemma}
\begin{proof}
 The statement~\eqref{eq:good-geometry-bound} follows from Lemmas~\ref{lem:fact0} and~\ref{lem:sufficient-for-gg}.
To prove~\eqref{eq:tau-exponential}, first note that $\tau_{n}+j$ is a stopping time for $\calF'_0, \calF'_1,\ldots$
for all $j \geq 1$. 
Also, constructing $X$ and $Y$ as described above, we have $G'_m = G_{(m+1)k}$. 
Let $A_m = G'_m \cap \{ V_{m+1} = 1\} \in \calF'_{m+1}$. Then by~\eqref{eq:good-geometry-bound}
we have that, for all $m \in\ZP$,
\begin{equation}
\label{eq:alpha-sq}
 \Pr ( A_{m+1} \mid \calF'_m ) \geq \alpha \Pr ( G_{(m+2)k} \mid \calF_{(m+1)k} ) \geq \alpha^2 , \as \end{equation}
Hence for $r \geq 1$,
\begin{align*}
 \Pr ( \tau_{n+1} - \tau_n \geq r+ 2 \mid \calF'_{\tau_n +1} ) 
& \leq \Exp \bigl[ \Pr ( A^\rc_{\tau_n+2+r} \mid \calF'_{\tau_n +1+r} ) \1 \{ \tau_{n+1} - \tau_n \geq r \} \bigmid \calF'_{\tau_n+1} \bigr] .\\
& \leq (1-\alpha^2) \Pr ( \tau_{n+1} - \tau_n \geq r \mid \calF'_{\tau_n+1} ) ,
\end{align*}
by~\eqref{eq:alpha-sq}. Then~\eqref{eq:tau-exponential} follows.
\end{proof}

Thus Lemma~\ref{lem:good-geometry} shows that the sequence
$\tau_1, \tau_2, \ldots$ does not terminate,
and its increments have exponentially bounded tails.
Consider the sequence $Y_{\tau_n}$. This is a Markov chain,
but its law is not translation invariant, due to the r\^ole of the origin.
The next section introduces a related process, whose increment law \emph{is} translation invariant,
and which therefore has i.i.d.~increments. In particular, it has a well-defined
\emph{radial drift} which entails ballisticity, and, crucially, it is close enough in behaviour to $Y_{\tau_n}$ to be able to deduce our theorems.

\section{A homogeneous process}
\label{sec:homogeneous}

For any fixed vector $\ell \in \Sp{d-1}$ we construct a homogeneous process in the direction $\ell$. Loosely speaking, it amounts to replacing the origin by a point at infinity in the direction $-\ell$. Let us give a precise definition.

For $\calX \subseteq \R^d$, we consider a semi-infinite cylinder with direction $-\ell$,
\begin{equation}\label{def:convell}
 {\conv}_\ell  (\calX) := \left\{  z -r \ell :  z \in \conv  \calX, \, r\geq 0 \right\}  .
\end{equation}
The set of $\ell$-\emph{admissible states} from $x \in \R^d$
with history $\calX \subseteq \R^d$ is 
\begin{align}  
\label{eq:admissiblell1}
\calA^\ell ( \calX ; x ) & := \cl \big\{ y \in B ( x; 1 ) : (x,y] \cap {\conv}_\ell ( \calX \cup \{ x \} ) = \emptyset \big\} \\
& = \cl \big( B(x;1) \setminus \cone ( x ; {\conv}_\ell ( \calX \cup \{ x\})  ) \big)  . 
\label{eq:admissiblell2}
\end{align}
We start with an observation relating the admissible states.

\begin{lemma}\label{lem:ell=hatx}
Suppose that $x \neq 0$ and $\calX \subset \R^d$. Then it holds that
\begin{equation} \label{eq:ObsA}
\calA^\ell(\calX;x)= \calA (\calX; x), \text{ for } \ell=\hat x.
\end{equation}
\end{lemma}
\begin{proof} 
When  $\ell = \hat x$, the origin  belongs to $ \conv_\ell (\calX \cup \{x\})$, which is a convex set.  Then $\conv (\calX \cup \{0,x\}) \subseteq
\conv_\ell (\calX \cup \{x\})$ and so comparison of~\eqref{eq:admissible1} with~\eqref{eq:admissiblell1} shows that $\calA^\ell(\calX;x) \subseteq \calA (\calX; x)$. 

Conversely,
consider the convex cone $C = \cone ( x ; \calX \cup \{0,x\}  )$;
the cone $C$ has vertex $x$ and contains $0$, so that the translate $C -\lambda x$ ($\lambda \geq 0$)
is contained in $C$. That is,  for any $z \in \conv(  \calX \cup \{x\})$ we have $z - \lambda x \in C$.
In other words,
$C$ is a convex cone that contains the cylinder ${\conv}_\ell ( \calX \cup \{ x\})$,
and hence $\cone ( x ; {\conv}_\ell ( \calX \cup \{ x\})  ) \subseteq C$.
Comparison of~\eqref{eq:admissible2} and~\eqref{eq:admissiblell2} shows that 
$\calA (\calX; x) \subseteq \calA^\ell(\calX;x)$, and the lemma is proved.
\end{proof}

We define the process $X^\ell := (X_0^\ell , X_1^\ell, \ldots )$
analogously to $X$.
Specifically, we set $\calX^\ell_{n,k} := \{ X^\ell_j : \max(1,n-k) \leq j \leq n-1 \}$,
 take $X^\ell_0 =0$, and suppose that, for $n \in \ZP$,
\[
 \Pr ( X^\ell_{n+1} \in A \mid X^\ell_0, X^\ell_1, \ldots, X^\ell_n ) = \int_A p^\ell ( y \mid \calX^\ell_{n,k} ; X^\ell_n ) \ud y,
\]
for all Borel sets $A \subseteq \R^d$, where  
\begin{equation} 
\label{def:walk8}
p^\ell (y \mid \calX ; x ) = \frac{1}{\vol{d} \calA^\ell ( \calX  ; x)} \1_{\calA^\ell ( \calX  ; x)}(y)
\end{equation} 
if $\vol{d} \calA^\ell ( \calX  ; x) > 0$.
This process is well defined, as shown by the following analogue 
of Lemma~\ref{lem:well-defined}; the proof is similar.

\begin{lemma}
\label{lem:well-defined-ell}
The process $X^\ell_0, X^\ell_1, X^\ell_2, \ldots$ is well defined, and for all $n \in \ZP$,
\[   \frac{\nu_d}{2} \leq 
	\vol{d} \calA^\ell ( \calX^\ell_{n,k} ; X^\ell_n)  \leq \nu_d . \]
\end{lemma}

A sequence $y_1,\ldots ,y_k \in \R^d$ is called an $\ell$-\emph{admissible path} from history 
$x_0, x_1, \ldots, x_k$ if 
$y_1  \in \calA^\ell ( x_0, \ldots, x_{k-1} ; x_k )$, 
$y_2 \in \calA^\ell( x_1, \ldots, x_{k-1} , x_k ; y_1 )$,
$y_3 \in \calA^\ell (x_2,\ldots, x_k, y_1; y_2)$,
and so on, up to
$y_k \in \calA^\ell (x_{k}, y_1 \ldots, y_{k-2} ; y_{k-1} )$.
Let $\calP^\ell ( x_0, x_1, \ldots, x_k )$ denote the set
of all $\ell$-admissible paths from history $x_0, x_1, \ldots, x_k$.

Let $\calF^\ell_n := \sigma (X_0^\ell, \ldots, X_n^\ell)$.
For $\delta \in (0,1/8)$ and $x \in \R^d$,
recall from~\eqref{eq:pi-def} that $\Pi^\ell (x)  =  \prod_{i=1}^k B  ( x +  \tfrac {i}{2} \ell ; \delta  )$. Also, define
\[
 \calG^\ell := \left\{ (x_0,\ldots,x_k) \in (\R^d)^{k+1} : x_k \neq 0, \text{ and } \Pi^\ell (x_k) \subseteq \calP^\ell ( x_0, \ldots, x_k) \right\}.
\]
For $n \geq k$, let $G^\ell_n \in \calF^\ell_n$ denote the event $G_n^\ell := \{ (X^\ell_{n-k},\ldots, X^\ell_n) \in  \calG^\ell \}$;
if $G^\ell_n$ occurs, we say that $X^\ell$ has \emph{good geometry} at time $n$.

The following analogue of Lemma~\ref{lem:fact1} is proved in the same way.

\begin{lemma}
\label{lem:fact1-ell} 
Let  $\alpha \in (0,1)$ be the constant appearing in Corollary~\ref{cor:f-bound}.
Then for all $n \geq k$ and all Borel $\fB \subseteq (\R^d)^k$, on the event $G^\ell_n$,
\[
\Pr (  (X^\ell_{n + 1},\ldots, X^\ell_{n+k}) \in \fB  \mid \calF^\ell_n ) \geq \alpha   \frac{  \vol{dk} \fB}{ ( \vol{d} B(0;\delta) )^k} \1 \{ \fB \subseteq \Pi^\ell ( X_n) \} .
\]
\end{lemma}

For $m \in \ZP$ define
\[ Y^\ell_{m} := (X^\ell_{m k},X^\ell_{m k +1},\ldots,X^\ell_{(m+1)k} ). \]
Now $Y^\ell = (Y_0^\ell, Y_1^\ell, \ldots)$ is a Markov chain and satisfies a version of Lemma~\ref{lem:Y-Markov}.
Moreover, we may assume that $Y^\ell$ is constructed along with its renewal
times $\tau^\ell_1, \tau^\ell_2, \ldots$, analogously to the construction of $Y'$ described
in Section~\ref{sec:renewal},
with $\Pi^\ell$ replacing $\Pi$, $f^\ell$ replacing $f$,
 and $\hat f^\ell$ replacing $\hat f$, where $f^\ell$ is defined by the analogue of~\eqref{eq:f-def}
with $p^\ell$ instead of $p$, and 
\begin{align*}
\hat f^\ell ( y_1 ,\ldots , y_k \mid x_0, \ldots, x_k ) 
= \frac{1}{1-\alpha} \left[ f^\ell (y_1,\ldots,y_k \mid x_0, \ldots, x_k) - \frac{\alpha \1_{\Pi^\ell (x_k)} (y_1,\ldots,y_k)}{(\vol{d} B (0;\delta))^k} \right] .
\end{align*}
Let $Y^\ell_{m,i}$ denote the $i$th component of $Y^\ell_m$,
and set $W^\ell_n := Y^\ell_{\tau^\ell_n,k+1} = X^\ell_{k \tau^\ell_n +k}$.

\begin{proposition}
\label{prop:iid-ell}
The sequence $(W^\ell_{n}; n \geq 1)$ is a homogeneous random walk,
that is, $(W^\ell_{n+1}-W^\ell_{n}; n \geq 1)$ 
is an i.i.d.~sequence. Moreover,
$\Exp \| W^\ell_{n+1}-W^\ell_{n} \| < \infty$ and 
\[ \Exp \left[ W^\ell_{n+1}-W^\ell_{n} \right] = u_{d,k} \ell ,\]
for a constant $u_{d,k}$ which does not depend on $\ell$. Finally,
the inter-renewal times $(\tau^\ell_{n+1} - \tau^\ell_n; n \geq 1)$
are i.i.d.~with $\Exp [ \tau^\ell_{n+1} - \tau^\ell_n ] = \lambda_{d,k}$
for a constant $\lambda_{d,k} \in (0,\infty)$ depending only on $d$ and $k$,
and 
such that, with $c >0$ the constant from Lemma \ref{lem:good-geometry},
\begin{equation}
\label{eq:tau-exponential-ell}
\Pr ( \tau^\ell_{n+1} - \tau^\ell_n \geq 2 r ) \leq  \re^{-c r} , \text{ for all } r \geq 0 .\end{equation}
\end{proposition}
\begin{proof}
By the renewal construction and the 
fact that for the $\ell$-process the transition function is translation invariant,
 \[ (Y^\ell_{\tau^\ell_j+1}- W^\ell_j, Y^\ell_{\tau^\ell_j +2}- W^\ell_j, \ldots, Y^\ell_{\tau^\ell_{j+1}}- W^\ell_j )
\text{ is an i.i.d.~sequence over $j \geq 1$}, \] 
where e.g.~$Y^\ell_{\tau^\ell_{j+1}}- W^\ell_j$ is the vector
with components $Y^\ell_{\tau^\ell_{j+1},i}- W^\ell_j$. Thus $W^\ell_{j+1}-W^\ell_j$ is also i.i.d.
Similarly, $\tau^{\ell}_{n+1} - \tau^\ell_n$ are i.i.d., so that 
$\Exp [ \tau^{\ell}_{n+1} - \tau^\ell_n ] =\lambda_{d,k}$
does not depend on $n$, and essentially the same argument as Lemma~\ref{lem:good-geometry}
gives the exponential bound~\eqref{eq:tau-exponential-ell}.

Next observe that
\[  \Exp  \left\|  W^\ell_{2}-W^\ell_{1} \right\| \leq k \Exp [ \tau^\ell_{2} - \tau^\ell_1 ] < \infty.  \]
The distribution of $ W^\ell_{2}-W^\ell_{1}$
is symmetric with respect to $\ell$, i.e., invariant
under any orthogonal transformation of $\R^d$ that leaves $\ell$ fixed.
Hence $\Exp [  W^\ell_{2}-W^\ell_{1} ] = u_{d,k} \ell$
for some $u_{d,k} \in \R$, which does not depend on $\ell$.
\end{proof}

\section{Coupling the processes}
\label{sec:coupling}

In this section we describe a coupling construction used to approximate the process $Y_m$ between times $\tau_n$ and $\tau_{n+1}$
by the process $Y^\ell_m$, where $\ell$ is fixed as $\ell = \hat Y_{\tau_n,k+1}$. 
We simultaneously construct the processes $Y$ and $Y^\ell$,
and their subsequent renewal times,  
essentially via the constructions described in Sections~\ref{sec:renewal} and~\ref{sec:homogeneous},
but with `maximal' exploitation of common randomness.

Our primary process we again denote by $Y$,
where $Y_n \in (\R^d)^{k+1}$, and we denote $Y_m = (X_{mk}, \ldots, X_{(m+1)k})$
in components, so the process $Y$ yields the process $X$.
	Let $Y_{m,i}$ denote the $i$th component of $Y_m$, so $Y_{m,i} = X_{mk+i-1}$.
	Given $\calF'_{\tau_n+1}$ (recall that $\tau_n+1$ is a stopping time), we will
 generate $Y_{\tau_n+2}, \ldots, Y_{\tau_{n+1}+1}$,
and, at the same time,
 generate $Y^\ell_{\tau_n+2}, \ldots, Y^\ell_{\tau_{n+1}+1}$,
where we couple the two processes and their renewal times in a maximal way (see below for formalities) 
starting at $Y^\ell_{\tau_n} = Y_{\tau_n}$
 and using the same underlying sequence $V_1, V_2, \ldots$. We stress that $\ell = \hat Y_{\tau_n,k+1}$ is kept fixed.

Before describing the coupling formally, we recall the following fact (see e.g.~\cite[p.~19]{lindvall}): 
If $X$ and $Y$ are random variables on $\R^p$ then there exists a \emph{maximal coupling}, i.e., a law on $(X,Y)$
such that $2 \Pr ( X \neq Y ) = \| \Pr ( X \in \, \cdot \, ) - \Pr ( Y \in \, \cdot \, ) \|_{\rm TV}$,
where $\| \,\cdot\, \|_{\rm TV}$ denotes total variation distance, which for measures $\mu_1$ and $\mu_2$ on $\R^p$ is 
defined by $\| \mu_1 - \mu_2 \|_{\rm TV} := \sup_B | \mu_1 (B) - \mu_2 (B) |$ where the supremum is over
Borel sets $B \subseteq \R^p$.

Here is the coupling construction.
As before, let
$V_1, V_2, \ldots$ be a sequence of i.i.d.~Bernoulli random variables
with $\Pr ( V_i = 1 ) = \alpha = 1 -\Pr (V_i = 0)$, where
$\alpha \in (0,1)$ is the constant in Corollary~\ref{cor:f-bound}.
The joint construction of $(Y_m, Y^\ell_m)$ will be adapted to the filtration $\calF'_{\tau_n+1}, \calF'_{\tau_n+2},\ldots$ (thus we
enlarge the previous filtration as necessary).
Recall that $G_{m} = \{ Y_{m} \in  \calG \}$ and $G_m^\ell = \{ Y_m^\ell \in \calG^\ell \}$. 

We begin by taking $Y^\ell_{\tau_n} = Y_{\tau_n}$.
Let $m \geq \tau_n$. If $Y_m \neq Y^\ell_m$, we generate
$Y_{m+1}, Y_{m+2}, \ldots$ and $Y^\ell_{m+1}, Y^\ell_{m+2}, \ldots$, and any associated renewal times, independently using
the constructions described previously in Sections~\ref{sec:renewal} and~\ref{sec:homogeneous}.
If $Y_m = Y^\ell_m$, then we generate $Y_{m+1}$ and $Y^\ell_{m+1}$ as follows.
\begin{enumerate}[leftmargin=*]
\item On the event $G_m^\rc \cap (G_m^\ell )^\rc$,
when neither process exhibits good geometry,
generate $Y_{m+1}$ and $Y^\ell_{m+1}$ via the maximal coupling of the corresponding marginal transition laws from $Y_m = Y^\ell_m$.
\item On the event $G_m \cap G_m^\ell$, when both processes exhibit good geometry, then:
 \begin{enumerate}[leftmargin=*]
\item If $V_{m+1} = 1$, 
declare that a renewal occurs for both processes ($\tau_{n+1}=\tau^\ell_{n+1} = m$) and 
set $Y_{m+1} = (Y_{m,k+1},Z_{k+1})$ and  $Y^\ell_{m+1} = (Y_{m,k+1},Z^\ell_{k+1})$ where
$Z_{k+1}$ and $Z^\ell_{k+1}$ are generated via a maximal coupling of 
the uniform laws on $\Pi ( Y_{m,k+1} )$ and $\Pi^\ell ( Y_{m,k+1} )$, respectively. Generate the subsequent trajectories 
of $Y$ and $Y^\ell$ independently.
\item If $V_{m+1} = 0$, 
set $Y_{m+1} = (Y_{m,k+1},Z_{k+1})$ and  $Y^\ell_{m+1} = (Y_{m,k+1},Z^\ell_{k+1})$ where now 
$Z_{k+1}$ and $Z^\ell_{k+1}$ are generated via a maximal coupling of 
the laws
corresponding to the densities $\hat f ( \, \cdot \mid Y_m)$ and $\hat f^\ell (\, \cdot \mid Y_m )$.
\end{enumerate}
\item On the event $G_m \sd G_m^\ell$ (where `$\sd$' denotes the symmetric difference),
generate  
$Y_{m+1}, Y_{m+2}, \ldots$ and $Y^\ell_{m+1}, Y^\ell_{m+2}, \ldots$, and any associated renewal times, independently.
\end{enumerate}

This construction gives $(Y_m, Y^\ell_m)$ for $m \geq \tau_n$
with the correct marginal distributions.
Let $E_n$ denote the event that the coupling `succeeds' between times $\tau_n$ and $\tau_{n+1}$, i.e.,
\begin{equation}
\label{def:En}
 E_n := \left\{ Y_m^\ell = Y_m \text{ for all } m \in \{\tau_n+1, \ldots, \tau_{n+1} \}, \text{ and } \tau^\ell_{n+1} = \tau_{n+1} \right\} .
\end{equation}
The effectiveness of the coupling is based on the following result,
whose proof we defer to the end of this section. Write $\log^2 n := (\log n)^2$.

\begin{proposition}
\label{prop:coupling_succeeds}
Let $E_n$ be as defined at~\eqref{def:En}.
There exists a constant $C \in \RP$ such that a.s., for all but finitely many $n \in \N$,
\[ \Pr (E_n \mid \calF'_{\tau_{n}+1} ) \geq 1 - \frac{C \log^2 n}{n} . \]
\end{proposition}

The fact that the coupling succeeds with high probability leads to the following
key result, which quantifies how well the homogeneous process approximates the
real process between renewal times.

\begin{corollary}
\label{cor:renewal-drift}
Let $W_n:= Y_{\tau_n +1,k+1}$. Then the following hold.
\begin{itemize}
\item[(i)]
Let $\lambda_{d,k}$ be the constant appearing in Proposition~\ref{prop:iid-ell}.
Then, for all $\gamma \in (0,1)$, 
\[ \lim_{n \to \infty} n^\gamma \left| \Exp [ \tau_{n+1} - \tau_n \mid \calF'_{\tau_n+1} ] - \lambda_{d,k}  \right|= 0, \as \]
\item[(ii)]
For all $p >0$, there is a constant $B \in \RP$ (depending on $p$, $d$, and $k$) such that
\[ \Exp [ \| W_{n+1} - W_n \|^p \mid \calF'_{\tau_n+1} ] \leq B , \as, \text{ for all } n \in \N. \]
\item[(iii)]
Let $u_{d,k}$ be the constant appearing in Proposition~\ref{prop:iid-ell}.
Then, for all $\gamma \in (0,1)$, 
\[ \lim_{n \to \infty} n^\gamma \left\| \Exp [ W_{n+1} - W_n \mid \calF'_{\tau_n+1} ] -u_{d,k} \hat W_n \right\|= 0, \as\]
\end{itemize}
\end{corollary}
\begin{proof}
For part~(i),  with $E_n$ as defined at~\eqref{def:En}, we have that
\begin{align*}
\Exp [ \tau_{n+1} - \tau_n \mid \calF'_{\tau_n+1} ] & = \Exp [ ( \tau_{n+1} - \tau_n ) \1 (E_n) \mid \calF'_{\tau_n+1} ]
+ \Exp [ ( \tau_{n+1} - \tau_n ) \1 (E^\rc_n) \mid \calF'_{\tau_n+1} ] .\end{align*}
Then, the (conditional) H\"older inequality implies that for all $p, q >1$
with $p^{-1} + q^{-1} = 1$,
\begin{align}
\label{eq:holder}
\left| \Exp [ ( \tau_{n+1} - \tau_n ) \1 (E^\rc_n) \mid \calF'_{\tau_n+1} ] \right|
& \leq  \left( \Exp [ (\tau_{n+1} - \tau_n)^p \mid  \calF'_{\tau_n+1} ] \right)^{1/p} \left( \Pr ( E^\rc_n \mid  \calF'_{\tau_n+1} )\right)^{1/q} \nonumber\\
& \leq  C_p    \left( \frac{\log^2 n}{n} \right)^{1/q} , \as,
\end{align}
for some  constant $C_p$ and all but finitely many~$n$,
by~\eqref{eq:tau-exponential} and Proposition~\ref{prop:coupling_succeeds}. In particular,
given $\gamma \in (0,1)$,
we may choose $q$ close enough to $1$ (and hence $p$ sufficiently large) so that this last bound is $o (n^{-\gamma})$.
On the other hand, on the event $E_n$ we have $\tau_{n+1} - \tau_n = \tau_{n+1}^\ell - \tau_n^\ell$ where $\ell = \hat Y_{\tau_n,k+1}$, so
\[  \Exp [ ( \tau_{n+1} - \tau_n ) \1 (E_n) \mid \calF'_{\tau_n+1} ]
=  \Exp [   \tau^\ell_{n+1} - \tau^\ell_n \mid \calF'_{\tau_n+1} ] - \Exp [ (\tau^\ell_{n+1} - \tau^\ell_n ) \1 (E_n^\rc) \mid \calF'_{\tau_n+1}] ,\]
where, similarly to above, $n^{\gamma} |  \Exp [ (\tau^\ell_{n+1} - \tau^\ell_n ) \1 (E_n^\rc) \mid \calF'_{\tau_n+1}] | \to 0$, a.s.
Moreover, by Proposition~\ref{prop:iid-ell}, $\Exp [   \tau^\ell_{n+1} - \tau^\ell_n \mid \calF'_{\tau_n+1} ] = \lambda_{d,k}$.
This establishes part~(i).

To prove part~(ii), observe first that, for all $n \in \N$,
\begin{equation}
\label{eq:renewal-bound}
 \| W_{n+1} - W_n \| \leq k ( \tau_{n+1} - \tau_n) , \as \end{equation}
Then the  statement in part~(ii) follows directly from~\eqref{eq:tau-exponential}.
For part~(iii), we have that 
\begin{align*}
\Exp [ W_{n+1} - W_n \mid \calF'_{\tau_n+1} ] & = \Exp [ ( W_{n+1} - W_n ) \1 (E_n) \mid \calF'_{\tau_n+1} ]
+ \Exp [ ( W_{n+1} - W_n ) \1 (E^\rc_n) \mid \calF'_{\tau_n+1} ] .\end{align*}
Then~\eqref{eq:renewal-bound} and an argument similar to~\eqref{eq:holder} shows that
\begin{align*}
n^{\gamma} \| \Exp [ ( W_{n+1} - W_n ) \1 (E^\rc_n) \mid \calF'_{\tau_n+1} ] \| \to 0, \as, 
\end{align*}
for all $\gamma \in (0,1)$.
On the other hand, on the event $E_n$ we have $W_{n+1} - W_n = W_{n+1}^\ell - W_n^\ell$ where $\ell = \hat Y_{\tau_n,k+1}$, so
\[  \Exp [ ( W_{n+1} - W_n ) \1 (E_n) \mid \calF'_{\tau_n+1} ]
=  \Exp [   W^\ell_{n+1} - W^\ell_n \mid \calF'_{\tau_n+1} ] - \Exp [ (W^\ell_{n+1} - W^\ell_n ) \1 (E_n^\rc) \mid \calF'_{\tau_n+1}] ,\]
where, again similarly to~\eqref{eq:holder}, we have that $n^{\gamma} \|  \Exp [ (W^\ell_{n+1} - W^\ell_n ) \1 (E_n^\rc) \mid \calF'_{\tau_n+1}] \| \to 0$, a.s.
Moreover, by Proposition~\ref{prop:iid-ell},
$\Exp [   W^\ell_{n+1} - W^\ell_n \mid \calF'_{\tau_n+1} ] = u_{d,k} \hat Y_{\tau_n,k+1}$.
To compare $\hat Y_{\tau_n,k+1} = \hat X_{\tau_n k + k}$
 to $\hat W_n = \hat Y_{\tau_n+1,k+1} = \hat X_{\tau_n k + 2k}$,
note that $\| X_{\tau_n k + k} - X_{\tau_n k + 2k} \| \leq k$.
For $x, y \in \R^d$ with $x \neq 0$ and $x+y \neq 0$, we have
\begin{align}
\label{eq:xy3} 
 \frac{ x + y}{\| x + y \|} - \frac{x}{\| x \|} & = \frac{y}{\| x + y\| } + \frac{x ( \| x \| - \| x + y \| )}{ \| x \| \| x + y\| }  .\end{align}
Applying~\eqref{eq:xy3},  
we see that 
$\| \hat Y_{\tau_n,k+1} - \hat W_n \| \leq 2 k/ \| W_n \|$. Since $\tau_n \geq n$ a.s., we have
from the final statement in Corollary~\ref{cor:exponential-bound} that $\liminf_{n \to \infty} n^{-1} \| W_n \| >0$.
Part~(iii) now follows.
\end{proof}

We can now complete the proof of  Theorem~\ref{thm:speed}. We use two results from the Appendix: Lemma~\ref{lem:norm-speed} 
which gives a law of large numbers for a process on $\RP$ under a drift and variance condition, and Lemma~\ref{lem:ballistic}
which implies ballisticity for a process on $\R^d$ given suitable radial drift asymptotics and a $\liminf$ speed bound.

\begin{proof}[Proof of Theorem~\ref{thm:speed}.]
By Corollary~\ref{cor:renewal-drift}(i) and~\eqref{eq:tau-exponential}, we may apply Lemma~\ref{lem:norm-speed} 
with $\zeta_n = \tau_n+1$ to obtain $\lim_{n \to \infty} n^{-1} \tau_n =\lambda_{d,k}$, a.s.
Moreover, by Corollary~\ref{cor:renewal-drift}(ii) and (iii), we may apply Lemma~\ref{lem:ballistic} to the process $W_n$ to obtain
$n^{-1} W_n \to u_{d,k} \ell$ for some random $\ell \in \Sp{d-1}$.
 
Let $J_n = \max \{ j \in \N : k \tau_{j}+ 2k \leq n \}$,
so that $k \tau_{J_n} + 2k \leq n < k \tau_{J_n +1} + 2k$.
Then by an inversion of the fact that $\tau_n/n \to \lambda_{d,k}$, 
we obtain $n^{-1} J_n \to ( k \lambda_{d,k} )^{-1}$, a.s. 
In particular, $J_n \to \infty$, a.s.
Then, since $W_{J_n} = Y_{\tau_{J_n} +1, k+1} = X_{k \tau_{J_n} + 2k}$, we have
\begin{equation}
\label{eq:W-X}
 \left\| \frac{X_n}{n} - \frac{W_{J_n}}{J_n} \cdot \frac{J_n}{n} \right\| \leq 
\frac{1}{n} \max_{k\tau_{J_n} + 2k \leq m < k\tau_{J_n+1} + 2k} \| X_m - X_{k \tau_{J_n} + 2k} \|
\leq \frac{1}{n}  \left( \tau_{J_{n+1}} - \tau_{J_n} \right) .\end{equation}
From~\eqref{eq:tau-exponential} and the Borel--Cantelli lemma we have that there exists $b < \infty$ such that
$\tau_{n+1} - \tau_n \leq b \log n$,
for all but finitely many $n$, a.s., and since $J_n = O(n)$ this implies that
$\tau_{J_{n+1}} - \tau_{J_n} \leq 2 b \log n$ for all but finitely many $n$, a.s.
Thus~\eqref{eq:W-X} yields 
\begin{equation}
\label{eq:v}
 \lim_{n \to \infty} n^{-1} X_n = \frac{u_{d,k}}{k \lambda_{d,k}} \ell, \as,
\end{equation}
which is the required a.s.~convergence result when we set $v_{d,k} := \frac{u_{d,k}}{k \lambda_{d,k}}$.
The bounded convergence theorem yields $n^{-1} \Exp \| X_n \| \to v_{d,k}$,
and the fact that $v_{d,k} > 0$ follows from Corollary~\ref{cor:exponential-bound}.
Finally, note that the law of $X$ is invariant
under orthogonal transformations of $\R^d$: for any orthogonal matrix $U$, the sequence $U X_0, UX_1, UX_2, \ldots$ has the same law
as the original $X_0, X_1, X_2, \ldots$, and so the $\ell$ in~\eqref{eq:v} satisfies
$U \ell \eqd \ell$, and the fact that $\ell$ is uniform on the sphere follows
from uniqueness of Haar measure.
\end{proof}

It remains to prove Proposition~\ref{prop:coupling_succeeds}.
To establish this result we need the following observations.
We denote by $\unif_d ( A )$ the uniform law on measurable $A \subseteq \R^d$ with $\vol{d} A \in (0,\infty)$.

\begin{lemma}
\label{lem:approximation}
There exists a constant $C \in \RP$ such that, for any $x \neq 0$,
\begin{align}
\label{eq:couple-uniforms}
\| \unif_{dk} ( \Pi ( x ) )  - \unif_{dk} ( \Pi^\ell ( x ) ) \|_{\rm TV} & \leq C \| \ell - \hat x \| .\end{align}
There exists a constant $C \in \RP$ such that for all $\ell \in \Sp{d-1}$ and all
 $x_0, \ldots, x_k$,
\begin{align}\label{eq:couple-both-bad}
 \| f ( \, \cdot \mid x_0, \ldots, x_k ) 
- f^\ell ( \, \cdot \mid x_0, \ldots, x_k ) \|_{\rm TV} \leq C \| \ell - \hat x_k \| + C \| x_k \|^{-1} ; \\
\label{eq:couple-hats}
 \| \hat f ( \, \cdot \mid x_0, \ldots, x_k ) 
- \hat f^\ell ( \, \cdot \mid x_0, \ldots, x_k ) \|_{\rm TV} \leq C \| \ell - \hat x_k \| + C \| x_k \|^{-1}  .\end{align}
Moreover, there exists $C \in \RP$ such that for all $\ell \in \Sp{d-1}$ and all
 $x_0, \ldots, x_k$,
\begin{align} 
\label{eq:ggdiff}
& {} | \Pr ( Y_{m+1} \in \calG \mid Y_m = (x_0, \ldots , x_k ) ) - \Pr ( Y^\ell_{m+1} \in \calG^\ell \mid Y_m^\ell = (x_0, \ldots, x_k ) ) | 
\nonumber\\
& {} \qquad\qquad\qquad\qquad\qquad\qquad\qquad\qquad\qquad {} \leq C \left( \| \ell - \hat x_k \| + \| x_k \|^{-1} \right).
\end{align}
\end{lemma}
\begin{proof}
First it is straightforward to show that for measurable $A_1, A_2 \subseteq \R^d$,
\begin{equation}
\label{eq:TV-uniforms}
\| \unif_d ( A_1 ) - \unif_d ( A_2 ) \|_{\rm TV} \leq \frac{2 \vol{d} ( A_1 \sd A_2 )}{ \vol{d}(  A_1 ) \wedge \vol{d} ( A_2 )} .\end{equation}
In particular, the bound~\eqref{eq:couple-uniforms}
follows from~\eqref{eq:TV-uniforms}
and 
the fact that $\vol{dk} ( \Pi^{\hat x} ( x ) \triangle \Pi^\ell ( x ) ) \leq C \| \ell - \hat x \|$,
since the centres of $B ( x + \frac{i}{2} \hat x ; \delta)$ and 
$B(x + \frac{i}{2} \ell ; \delta)$ are at distance at most $\frac{k}{2} \| \ell - \hat x \|$.

Next we claim that there exists a constant $C \in \RP$ such that for all $\ell \in \Sp{d-1}$, all $x \in  \R^d \setminus B(0;k+1)$ and all $\calX \subset B(x;k)$ of cardinality $k$ such that $\calA (\calX; x)$ and $\calA^\ell (\calX;x)$ have volume not smaller than $\frac12 \nu_d$, 
\begin{equation}\label{eq:couplage}
\| p ( \, \cdot \, \mid \calX ; x) -  p^\ell ( \, \cdot \, \mid \calX ; x ) \|_{\rm TV} \leq C \| \ell - \hat x \| .
\end{equation}
By Lemma~\ref{lem:ell=hatx}, $ p ( \, \cdot \, \mid \calX ; x) =  p^\ell ( \, \cdot \, \mid \calX ; x ) $ for $\ell=\hat x$, so \eqref{eq:couplage} is equivalent to
\[
\| p^{\hat x} ( \, \cdot \, \mid \calX ; x ) -  p^\ell ( \, \cdot \, \mid \calX ; x ) \|_{\rm TV} \leq C \| \ell - \hat x \| .
\]
Moreover, it follows from~\eqref{eq:TV-uniforms} that there exists a constant $C$ such that for all $A_1, A_2 \subseteq B(0;1)$ with volume not smaller than $\frac12 \nu_d$, 
\begin{equation}
\| \unif_d ( A_1 ) - \unif_d ( A_2 ) \|_{\rm TV} \leq C \vol{d} ( A_1 \triangle  A_2) .
\label{eq:MPS}
\end{equation}
It remains to estimate the volume of 
$\calA^\ell(\calX;x) \triangle  \calA^{\hat x}(\calX;x)$. Taking a parametrization of the segment from $\hat x$ to $\ell$, say $\ell(\lambda)=\hat x + \lambda (\ell-\hat x)$, $\lambda \in [0,1]$,
we can control  the derivative of the volume by the surface measure of the boundary of  the admissible set,
\begin{equation} \label{eq:surf}
\left| \frac{\ud}{\ud\lambda} {\rm Vol}_d\big( \calA^{\ell(\lambda)}(\calX;x) \triangle  \calA^{\hat x}(\calX;x)\big)\right|  \leq 
{\rm Surf}\left( \partial \calA^{\ell(\lambda)}(\calX;x)\right) \times \|\ell-\hat x\|.
\end{equation}
But the set $ \cone ( x ; {\conv}_\ell ( \calX \cup \{ x \} )  ) $ in \eqref{eq:admissiblell1} has a finite number of hyperplanar faces, uniformly bounded for a fixed $k$. Since 
$ \calA^{\ell(\lambda)}(\calX;x)$ has diameter less than $1$, we conclude that the  surface term is bounded, and further, that \eqref{eq:couplage} holds.

We claim that~\eqref{eq:couple-both-bad} follows from~\eqref{eq:couplage}.
Indeed, fix $x_0, \ldots, x_k$ and 
let $Z = (Z_1, \ldots, Z_k)$ be a random vector in $(\R^d)^{k}$
with 
\[ \Pr ( Z_1 , \ldots, Z_k \in \fB ) = \int_{\fB} f ( y_1, \ldots, y_k \mid x_0, \ldots, x_k) \ud y_1  \cdots \ud y_k ,\]
and let $Z^\ell = (Z^\ell_1, \ldots, Z^\ell_k)$ have the same distribution but with $f^\ell$ instead of $f$.
To estimate $\| Z - Z^\ell \|_{\rm TV}$ we couple $Z$ and $Z^\ell$ component by component.
Then~\eqref{eq:couplage} shows that we can couple $Z_1$ and $Z_1^\ell$ such that $\Pr ( Z_1 \neq Z_1^\ell ) \leq C \| \ell - \hat x_k \|$. Given $Z_1 = Z_1^\ell = y_1$, the conditional densities of $Z_2$ and $Z_2^\ell$
are $p ( \, \cdot \mid x_1, \ldots, x_k ; y_1)$ and $p^\ell ( \, \cdot \mid x_1, \ldots, x_k ; y_1)$,
respectively, so by~\eqref{eq:couplage} we may again couple so that $\Pr ( Z_2 \neq Z_2^\ell \mid Z_1 = Z_1^\ell ) \leq 
C \| \ell - \hat y_1 \| \leq C \| \ell - \hat x_k \| + C \| x_k \|^{-1}$. Iterating this argument yields
a coupling of $Z$ and $Z^\ell$ that fails with probability at most $C \| \ell - \hat x_k \| + C \| x_k \|^{-1}$,
which implies the total variation bound in~\eqref{eq:couple-both-bad}.

Next, we claim that~\eqref{eq:couple-hats} follows from~\eqref{eq:couple-uniforms}
and~\eqref{eq:couple-both-bad}. Indeed, by the definitions of $\hat f$ and $\hat f^\ell$,
\begin{align*}
& {} \qquad {} \left| \int_B \hat f( y_1, \ldots, y_k \mid x_0, \ldots, x_k) \ud y_1 \cdots \ud y_k
- \int_B \hat f^\ell( y_1, \ldots, y_k \mid x_0, \ldots, x_k)\ud y_1 \cdots \ud y_k \right| \\
& \leq \frac{1}{1-\alpha} \left| \int_B  f( y_1, \ldots, y_k \mid x_0, \ldots, x_k) \ud y_1 \cdots \ud y_k
- \int_B f^\ell( y_1, \ldots, y_k \mid x_0, \ldots, x_k)\ud y_1 \cdots \ud y_k \right| \\
& {} \qquad {} + \frac{1}{\alpha} \| \unif_{dk} ( \Pi(x_k) ) - \unif_{dk} (\Pi^\ell(x_k)) \|_{\rm TV} ,\end{align*}
which gives the result.

We now turn to the proof of~\eqref{eq:ggdiff}. It is sufficient to consider the case $\|x_k\| \geq 2k$. 
We decompose
\begin{align*}  
T & =  | \Pr ( Y_{m+1} \in \calG \mid Y_m = (x_0, \ldots , x_k ) ) - \Pr ( Y^\ell_{m+1} \in \calG^\ell \mid Y_m^\ell = (x_0, \ldots, x_k ) ) |  \\  
&\leq 
 | \Pr ( Y_{m+1} \in \calG \mid Y_m = (x_0, \ldots , x_k ) ) -  \Pr ( Y^\ell_{m+1} \in \calG \mid Y_m^\ell = 
 (x_0, \ldots, x_k ) ) |  \\ 
 & {} \qquad {} + | 
 \Pr ( Y^\ell_{m+1} \in \calG \mid Y_m^\ell =  (x_0, \ldots, x_k ) ) -
  \Pr ( Y^\ell_{m+1} \in \calG^\ell \mid Y_m^\ell = (x_0, \ldots, x_k ) ) | \\  
  & =:  T_1+T_2.
\end{align*}
Here $T_1 \leq C \| \ell - \hat x_k \| + C \| x_k \|^{-1}$ by~\eqref{eq:couple-both-bad}.
For the other term we see from~\eqref{eq:MPS} that
\[
T_2 \leq C \vol{dk} \big( ( \calG  \triangle  \calG^\ell )(x_k)\big),
\]
where we have used the notation
\[ 
\calH (x)=\{ (y_1,\ldots,y_k):  (x,y_1,\ldots,y_k)\in \calH \}, \text{ for } \calH \subset \R^{d(k+1)}, \, x \in \R^d.
\]
It remains to prove that
\begin{equation} 
\label{eq:surf2}
 \vol{dk} \big( ( \calG  \triangle \calG^\ell )(x_k)\big)     \leq C ( \|\ell - \hat x_k\|+  \|x_k\|^{-1})   ,
\end{equation}
for some constant $C$. 
For $k=1$ we simply use~\eqref{eq:surf} to conclude~\eqref{eq:surf2}. For general $k$, we observe that
the set $ \calG^\ell (x_k)$  has a smooth boundary with bounded surface measure in $\R^{kd}$. Then
\begin{align*}
 \vol{dk} \big( ( \calG  \triangle  \calG^\ell )(x_k)\big)   &  \leq C  \sup\{ \|\ell - \hat y_i\| : (y_1,\ldots, y_k) \in \calG(x_k) \} ,
\end{align*}
which yields~\eqref{eq:surf2}. This ends the proof.
\end{proof}

\begin{proof}[Proof of Proposition~\ref{prop:coupling_succeeds}.]
For $r \in \N$ define the event
\[ F_{n,r} := \bigcap_{m = \tau_n +1}^{\tau_n +r } \left( \{ Y_m^\ell = Y_m \} \cap ( G_m \triangle G_m^\ell )^\rc \right) .\]
Note that $F_{n,\tau_{n+1}-\tau_n} \subseteq E_n$.
Then, for $r_n := \lceil A \log n \rceil$ where $A >0$,
\begin{align*}
\Pr ( E_n^\rc \mid \calF'_{\tau_n+1} ) & \leq \Pr ( E_n^\rc, \, \tau_{n+1} - \tau_n \leq r_n \mid \calF'_{\tau_n+1} ) + \Pr (  \tau_{n+1} - \tau_n > r_n \mid \calF'_{\tau_n+1} ) \\
& \leq \Pr ( F_{n,r_n}^\rc \mid \calF'_{\tau_n+1} ) + \Pr (  \tau_{n+1} - \tau_n > r_n \mid \calF'_{\tau_n+1} ).\end{align*}
By~\eqref{eq:tau-exponential}, we may (and do) choose $A$ sufficiently large so that $\Pr (  \tau_{n+1} - \tau_n > r_n \mid \calF'_{\tau_n+1} ) \leq C/n$, a.s.
The result now follows from the claim that there is a constant $C \in \RP$ such that, for all $n \in \N$ and all $r \in \N$,
\begin{equation}
\label{eq:finite-coupling}
\Pr ( F_{n,r}^\rc \mid  \calF'_{\tau_n+1} ) \leq \frac{ C r^2}{\| Y_{\tau_n+1,k+1} \|} , \as ,
\end{equation}
and the fact that, by Corollary~\ref{cor:exponential-bound}, $\| Y_{\tau_n+1,k+1} \| = \| X_{(\tau_n +2)k} \| > c \tau_n$ a.s.~for some $c>0$ and all but finitely many~$n$,
with the simple bound $\tau_n \geq n$ a.s.
It thus remains to prove the claim~\eqref{eq:finite-coupling}. Since $F_{n,r+1} \subseteq F_{n,r}$ for $r \in \N$, we have
\begin{align*} \Pr ( F_{n,r+1}^\rc \mid  \calF'_{\tau_n+1} ) & \leq \Pr ( F_{n,r+1}^\rc \cap F_{n,r}  \mid  \calF'_{\tau_n+1} ) 
+ \Pr ( F_{n,r}^\rc   \mid  \calF'_{\tau_n+1} )  \\
& = \Exp \big[ \Pr ( F_{n,r+1}^\rc \mid \calF'_{\tau_n+r} ) \1 ( F_{n,r} ) \bigmid \calF'_{\tau_n+1} \big] 
+ \Pr ( F_{n,r}^\rc   \mid  \calF'_{\tau_n+1} ) ,\end{align*}
so to verify~\eqref{eq:finite-coupling} it suffices to prove that, for all $n \in \N$ and all $r \in \N$,
\begin{equation}
\label{eq:one-step}
\Pr ( F_{n,r+1}^\rc \mid \calF'_{\tau_n+r} ) \leq \frac{Cr}{\| Y_{\tau_n+1,k+1} \|} \text{ on } F_{n,r} .
\end{equation}
To this end, note that, on $F_{n,r}$,
\begin{align*}
\Pr ( F_{n,r+1}^\rc \mid  \calF'_{\tau_n+r} ) &
\leq \Pr ( Y_{\tau_n+r+1} \neq Y_{\tau_n+r+1}^\ell  \mid  \calF'_{\tau_n+r} )
+ \Pr ( G_{\tau_n + r+1} \triangle G_{\tau_n + r+1}^\ell \mid  \calF'_{\tau_n+r} ) .\end{align*}
On $F_{n,r}$ we have that either (i) $G_{\tau_n+r}^\rc \cap (G_{\tau_n+r}^\ell)^\rc$ occurs,
in which case 
\begin{align*}  \Pr ( Y_{\tau_n+r+1} \neq Y_{\tau_n+r+1}^\ell  \mid  \calF'_{\tau_n+r} ) 
& \leq \| f ( \, \cdot \mid Y_{\tau_n +r } ) 
- f^\ell ( \, \cdot \mid Y_{\tau_n +r} ) \|_{\rm TV} \\
& \leq C   \| \ell - \hat Y_{\tau_n +r, k+1} \| + C \|  Y_{\tau_n +r, k+1} \|^{-1}  ,\end{align*}
by~\eqref{eq:couple-both-bad}, or (ii) $G_{\tau_n+r}  \cap G_{\tau_n+r}^\ell$ occurs,
in which case
\begin{align*} 
\Pr ( Y_{\tau_n+r+1} \neq Y_{\tau_n+r+1}^\ell  \mid  \calF'_{\tau_n+r} )
& \leq \alpha \| \unif_{dk} ( \Pi ( Y_{\tau_n,k+1} ) )  - \unif_{dk} ( \Pi^\ell ( Y_{\tau_n,k+1} ) )  \|_{\rm TV} \\
& {} \qquad {} + (1-\alpha) \| \hat f ( \, \cdot \mid  Y_{\tau_n +r} ) - \hat f^\ell ( \, \cdot \mid Y_{\tau_n +r} ) \|_{\rm TV}
\\
& \leq C \| \ell - \hat Y_{\tau_n+r,k+1} \| + C \| Y_{\tau_n+r,k+1} \|^{-1} ,
\end{align*}
by~\eqref{eq:couple-uniforms} and~\eqref{eq:couple-hats}. Moreover, on $F_{n,r}$, by~\eqref{eq:ggdiff},
\[ \Pr ( G_{\tau_n + r+1} \triangle G_{\tau_n + r+1}^\ell \mid  \calF'_{\tau_n+r} ) 
\leq C \| \ell - \hat Y_{\tau_n+r,k+1} \| + C \| Y_{\tau_n+r,k+1} \|^{-1}   .
\]
Since $\ell = \hat Y_{\tau_n+1,k+1}$ and
$\|  Y_{\tau_n+r,k+1} - Y_{\tau_n+1,k+1} \| \leq r k$, we thus obtain~\eqref{eq:one-step}.
This completes the proof.
\end{proof}

\section{The planar case with unit memory}
\label{sec:memory-1}

This section is devoted to the proof of Theorem~\ref{thm:memory-one},
and so we take $d=2$ and $k=1$ throughout this section. 
For $n \in \N$, denote by $\theta_n \in [0,\pi]$ the magnitude of the interior angle of $\conv \{0,X_{n-1},X_n\}$ at $X_n$:
see Figure~\ref{fig2}.

\begin{figure}
\centering
\begin{tikzpicture}
\draw[fill=black] (0,4) circle (.5ex); 
\draw[line width=0.3mm] (0.0,4.0) -- (4.0,4.0); 
\node at (0.2,4.3) {$\0$};
\draw[line width=0.3mm] (4.0,4.0) -- (2.5,2.0);
\draw[fill=black] (4,4) circle (.5ex); 
\node at (4.5,3.5) {$X_n$};
\draw[fill=black] (2.5,2) circle (.5ex); 
\node at (3,1.5) {$X_{n-1}$};
\draw (3,4) arc (180:233:1);
\node at (2.8,3.4) {$\theta_n$};
\draw[->,dashed] (4.5,4) arc (0:233:0.5);
\node at (4,4.8) {$\phi$};
\draw[dotted,line width=0.3mm,->] (4.0,4.0) -- (5.5,4.0); 
\end{tikzpicture}
\qquad\qquad
\begin{tikzpicture}
\draw[fill=black] (0,4) circle (.5ex); 
\draw[line width=0.3mm] (0.0,4.0) -- (4.0,4.0);
\node at (0.2,4.3) {$\0$}; 
\draw (1,4) arc (0:19:1);
\node at (1.5,4.2) {$\alpha_{n+1}$};
\draw[line width=0.3mm] (4.0,4.0) -- (2.5,2.0);
\draw[fill=black] (4,4) circle (.5ex); 
\node at (4.5,3.5) {$X_n$};
\draw[fill=black] (2.5,2) circle (.5ex); 
\node at (3,1.5) {$X_{n-1}$};
\draw[fill=black] (4.5,5.5) circle (.5ex); 
\node at (5,5.8) {$X_{n+1}$};
\draw[line width=0.3mm] (4.0,4.0) -- (4.5,5.5);
\draw[line width=0.3mm] (0.0,4) -- (4.5,5.5); 
\draw (3,4) arc (180:233:1);
\node at (2.8,3.4) {$\theta_n$};
\draw (4,5.34) arc (210:263:0.5);
\node at (3.8,4.9) {$\theta_{n+1}$};
\draw[dotted,line width=0.3mm,->] (4.0,4.0) -- (5.5,4.0); 
\draw (3.6,4.0) arc (180:70:0.4);
\draw (3.65,4.0) arc (180:70:0.35);
\draw (3.75,3.65) arc (233:432:0.43);
\node at (4.8,4.5) {$\phi_{n+1}$};
\end{tikzpicture}
\caption{The definition of $\theta_n$ (\emph{left}) and
the construction of $\theta_{n+1}$ (\emph{right}).
In the right-hand diagram, the double-ruled angle is $2\pi - \theta_n - \phi_{n+1}$.}
\label{fig2}
\end{figure}
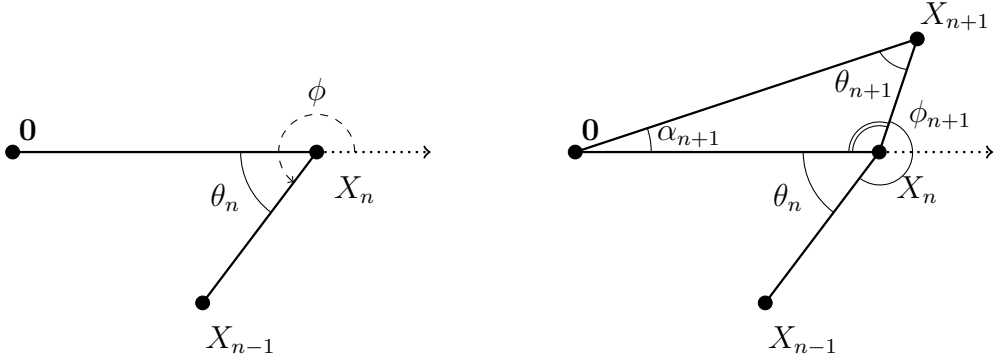

First we express the `local drift' in terms of $\theta_n$.

\begin{lemma}
\label{lem:local-drift}
Let $d=2$ and $k=1$. Then for $n \in \N$,
\[ \Exp [ ( X_{n+1} - X_n) \cdot \hat X_n ] = \Exp \left[ \frac{2\sin \theta_n}{6\pi - 3 \theta_n} \right] .\]
\end{lemma}
\begin{proof}
Suppose that $n \in \N$; note that $X_n \neq 0$ a.s.
Given $X_n$ and $X_{n-1}$, let $(r,\phi)$ be polar coordinates
with origin $(r=0)$ at $X_n$, $\phi =0$ in the direction $\hat X_n$,
and oriented so that $X_{n-1}$ is at angle $\phi = \theta_n - \pi$ relative to $X_n$.
See the left-hand part of Figure~\ref{fig2}.

The area of the disk sector on which $X_{n+1}$ is uniformly distributed is
$\pi - (\theta_n/2)$, so 
\begin{align*}
\Exp [ ( X_{n+1} - X_n) \cdot \hat X_n  \mid X_n, X_{n-1} ] 
& = \frac{2}{2\pi - \theta_n} \int_0^1   \int_{\theta_n-\pi}^{\pi} ( r \cos \phi ) \ud \phi \, r  \ud r \\
& = \frac{2 \sin \theta_n}{6\pi - 3\theta_n} ,\end{align*}
which gives the result.
\end{proof}

\begin{lemma} \label{lem:ato0}
Let $d=2$ and $k=1$. Then for $n \in \N$,
\[ \theta_{n+1} = | ( 2\pi - \theta_n) U_{n+1} - \pi | - \alpha_{n+1} ,\]
where $U_1, U_2, \ldots$ are i.i.d.~$U[0,1]$ random variables,
and $\alpha_n \to 0$ a.s.~as $n \to \infty$.
\end{lemma}
\begin{proof}
Let $n \in \N$.
Given $X_n$ and $X_{n-1}$, once more use 
the polar coordinates described in the proof of Lemma~\ref{lem:local-drift}.
The angle $\phi$ 
of $X_{n+1}$ in these coordinates
is $\theta_n + \phi_{n+1} -\pi$, where $\phi_{n+1}$,
the angle between vectors $X_nX_{n-1}$ and $X_n X_{n+1}$ measured outside the convex hull,
 is uniformly
 distributed on $[0, 2\pi -\theta_n]$; say $\phi_{n+1} = (1 - U_{n+1}) (2\pi -\theta_n)$
for $U_{n+1} \sim U[0,1]$ independent of $X_n, X_{n-1}$. See  the right-hand part of Figure~\ref{fig2}.
The angle at $X_{n+1}$ 
in the triangle $T$ with vertices $0, X_n, X_{n+1}$
 is $\theta_{n+1}$;
denote the angle at $0$ in $T$ by $\alpha_{n+1}$.
The angle at $X_n$ in $T$  is either
$\theta_n + \phi_{n+1}$ (if $\theta_n + \phi_{n+1} \leq \pi$)
or $2\pi - \theta_n - \phi_{n+1}$ (if $\theta_n + \phi_{n+1} > \pi$).
In the first case 
\[ \theta_{n+1} = \pi - \theta_n - \phi_{n+1} - \alpha_{n+1} ,\]
and in the second case 
\[ \theta_{n+1} = \pi - (2\pi - \theta_n - \phi_{n+1} ) - \alpha_{n+1} = \theta_n + \phi_{n+1} - \pi - \alpha_{n+1}.\]
Combining these we get
\[ \theta_{n+1} = | \theta_n +\phi_{n+1} - \pi | - \alpha_{n+1} 
= |    (2\pi - \theta_n ) U_{n+1} - \pi   | - \alpha_{n+1} 
.\]
Since $\| X_{n+1} - X_n \| \leq 1$, it is not hard to see that $0 \leq \alpha_{n+1} \leq C ( 1 + \| X_n \| )^{-1}$,
and this tends to $0$ since $\| X_n \| \to \infty$ by Corollary~\ref{cor:exponential-bound}.
\end{proof}

\begin{lemma}
\label{lem:theta-convergence}
We have that
$\theta_n \tod \theta$ as $n \to \infty$
where $\theta \in [0,\pi]$ has the distribution uniquely determined by the 
distributional fixed-point equation
\begin{equation}
\label{eq:fixed-point}
 \theta \eqd | ( 2\pi - \theta) U - \pi | , \, \theta \in \R ,\end{equation}
where $U \sim U[0,1]$ is independent of the $\theta$ on the right.
Moreover, the random variable $\theta$ has probability density
given by
\begin{equation}
\label{eq:theta-density}
 f(t) = \frac{2}{3\pi^2} (2 \pi -t ) , \text{ for } t \in [0,\pi] .
\end{equation}
\end{lemma}
\begin{remark}
It is not hard to check that~\eqref{eq:theta-density} provides a solution to~\eqref{eq:fixed-point}:
see the proof below, which also establishes uniqueness.
To come up with~\eqref{eq:theta-density} in the first place, one can deduce 
that the density $f$ of $\theta$ solving~\eqref{eq:fixed-point} satisfies the
differential equation 
$(\pi + t)f'(t) = - f(\pi - t)$ for all $t \in (0 , \pi)$ (by differentiating \eqref{eq:equadiff} below), and we observed that a linear $f$ solves this.
\end{remark}
\begin{proof}[Proof of Lemma~\ref{lem:theta-convergence}.]
Define
\[
  T(x,u) :=  |(2\pi-x)u-\pi| .
\]
Then the fixed-point equation~\eqref{eq:fixed-point} reads $T ( \theta, U) \eqd \theta$, while
Lemma~\ref{lem:ato0} shows that $\theta_{n+1} = T(\theta_n, U_{n+1}) - \alpha_{n+1}$.

Let $\theta$ satisfy~\eqref{eq:fixed-point}. Then clearly $\theta \geq 0$, a.s. Moreover,
for any $t \geq 0$,
\begin{equation}
\label{eq:T-theta}
 \Pr ( \theta > t ) = \Pr ( T (\theta, U ) > t )
= \Pr ( (2\pi - \theta) U > \pi + t ) + \Pr ( (2\pi - \theta) U < \pi - t ) .\end{equation}
In particular, taking $t = r \pi$ for $r \in \N$, using the fact that $\theta \geq 0$ and $U \in [0,1]$,
\begin{align*} \Pr (\theta > r \pi ) & \leq \Pr ( 2 \pi U > (1+r) \pi ) + \Pr ( \theta \geq 2 \pi , \, ( \theta - 2\pi) U > (r-1) \pi ) \\
& \leq 0 + \Pr ( \theta > (r+1) \pi ) ,\end{align*}
so that $\Pr ( \theta > r\pi ) = \Pr ( \theta > (r+1) \pi)$ and hence
$\Pr ( \theta > \pi ) = \lim_{r \to \infty} \Pr ( \theta > r \pi ) = 0$. Thus any (finite)
solution $\theta$ to~\eqref{eq:fixed-point} has $\theta \in [0,\pi]$, a.s.

Define a Markov 
transition operator $Q$ on state-space $[0,\pi]$ by $Q(x,A) = \Pr ( T (x, U) \in A )$ where $U \sim U[0,1]$, $x \in [0,\pi]$ and measurable $A \subseteq [0,\pi]$.
Then~\eqref{eq:fixed-point} is equivalent to the statement that
$\Exp Q (\theta, A) = \Pr ( \theta \in A)$ for all measurable $A \subseteq [0,\pi]$, i.e.,
the distributional solutions to~\eqref{eq:fixed-point} are precisely the invariant measures of~$Q$.
Note also that
\[ Q (x, \ud y) \geq \Pr ( \pi - (2\pi-x) U \in \ud y ) \geq \frac{1}{2\pi} \ud y , \text{ for all } x,y \in (0,\pi), \]
so that for any measurable $A \subseteq [0,\pi]$, $\inf_x Q(x,A) \geq \frac{1}{2} g(A)$
where $g$ is uniform measure on $[0,\pi]$.
This is a version of the Doeblin condition, and it ensures (see e.g.~\cite[Theorem~16.0.2, p.~394]{mt}) that $Q$ is uniformly ergodic:
there is a unique invariant measure $\mu$ such that 
 $\sup_{x \in [0,\pi]} \| Q^m (x , \, \cdot \, ) - \mu ( \, \cdot \, ) \|_{\rm TV} \to 0$ as $m \to \infty$.
In particular~\eqref{eq:fixed-point} has a unique distributional solution.
Moreover, if $\calP[0,\pi]$ denotes the set of probability measures on $[0,\pi]$, then
\begin{equation}
\label{eq:DF} \lim_{m\to\infty} \sup_{\nu \in \calP[0,\pi]} \| \nu Q^m - \mu\|_{\rm TV} = 0.
 \end{equation}

 Let $\nu_n$ denote the law of $\theta_n$.
If $(\psi_k, \psi_{k+1},\ldots)$ is the Markov chain started from $\psi_k = \theta_k$
and with evolution $\psi_{k+m+1} = T ( \psi_{k+m}, U_{k+m+1} )$, then $(\psi_{k},\psi_{k+1},\ldots)$ lives on the same probability space as $(\theta_{k},\theta_{k+1},\ldots)$,
and $\psi_{k+m}$ has law $\nu_k Q^m$.
Let $\rho$ denote the L\'evy--Prokhorov metric on distributions.
Then
\begin{align}
\label{eq:rho}
\rho( \nu_{k+m} , \mu ) \leq \rho (\nu_{k+m} , \nu_{k} Q^m ) + \rho (  \nu_k Q^m , \mu ) .\end{align}
Here by~\eqref{eq:DF} we can choose $m$ sufficiently large so that $\rho (  \nu_k Q^m , \mu ) \leq \|  \nu_k Q^m - \mu \|_{\rm TV} \leq \eps$ for all $k$.
On the other hand,
we see that $| T(x,u) - T(y,u) | \leq | x - y|$,
so 
\[ | \psi_{k+\ell+1} - \theta_{k+\ell+1} | \leq | T ( \psi_{k+\ell}, U_{k+\ell+1} ) - T( \theta_{k+\ell} , U_{k+\ell+1} ) - \alpha_{k+\ell+1} |
\leq | \psi_{k+\ell} - \theta_{k+\ell} | + \alpha_{k+\ell+1} ,\]
which shows that $| \psi_{k + m} - \theta_{k+m} | \leq \sum_{j=k}^{k+m} \alpha_j$.
Thus, since $\alpha_j \to 0$, 
for any $m$ we have $\lim_{k \to \infty} | \psi_{k+m} - \theta_{k+m} | =0$, a.s.,
and so $\rho (\nu_{k+m} , \nu_{k} Q^m ) \to 0$ as $k \to \infty$.
Thus in~\eqref{eq:rho} we may take both $m$ and $k$ large to see that 
$\lim_{n \to \infty} \rho (\nu_n, \mu)=0$.
Thus $\theta_n$ converges in law to $\mu$, the unique  distributional  solution to~\eqref{eq:fixed-point}. 

It remains to identify the law $\mu$. To this end, we check that if $\theta$ has density $f$ as given by~\eqref{eq:theta-density}, then
\begin{align}\nn
\Pr ( T (\theta, U) \leq t ) &  = \Exp \left[ \frac{2t}{2\pi -\theta} \1 \{ \theta \leq \pi - t \} \right] + \Exp \left[ \frac{\pi + t -\theta}{2\pi -\theta} \1 \{ \theta > \pi - t \} \right] \\ \label{eq:equadiff}
& = 2t \int_0^{\pi-t} \frac{f(y)}{2\pi -y} \ud y + \int_{\pi-t}^\pi \frac{(\pi + t -y)f(y)}{2\pi -y} \ud y \\ \nn
& =  \frac{4t (\pi - t)}{3\pi^2} + \frac{2}{3\pi^2} \int_{\pi-t}^\pi ( \pi + t - y) \ud y \\ \nn
& =  \frac{4t (\pi - t)}{3\pi^2} + \frac{t^2}{\pi^2} = \frac{t (4 \pi - t)}{3\pi^2} ,
\end{align}
which is $\int_0^t f(s) \ud s$. Hence $f$ provides a solution to 
the distributional equation~\eqref{eq:fixed-point}.  
 \end{proof}

Finally, we can complete the proof of Theorem~\ref{thm:memory-one}.

\begin{proof}[Proof of Theorem~\ref{thm:memory-one}.]
By  Lemmas~\ref{lem:local-drift} and~\ref{lem:theta-convergence} and the bounded convergence theorem,
\[ \lim_{n \to \infty} \Exp [ ( X_{n+1} - X_n ) \cdot \hat X_n ] = \Exp \left[ \frac{2\sin \theta}{6\pi - 3 \theta} \right] ,\]
where $\theta$ has the density given by~\eqref{eq:theta-density}. Then we compute
\[ \Exp  \left[ \frac{2\sin \theta}{6\pi - 3 \theta} \right]
= \frac{4}{9\pi^2} \int_0^\pi   \sin t \ud t = \frac{8}{9\pi^2}  .\]
Comparison of Theorem~\ref{thm:speed} and Lemma~\ref{lem:norms} shows that this quantity is
indeed $v_{2,1}$, which ends the proof of Theorem~\ref{thm:memory-one}.
\end{proof}

\appendix

\section{Auxiliary results: speeds and directions}
\label{sec:appendix}

The next result, which will be our tool for establishing ballisticity,
is an important ingredient in the proof of Theorem~\ref{thm:speed}.

\begin{lemma}
\label{lem:ballistic}
Let $d \in \N$. Let $\xi_0, \xi_1, \xi_2, \ldots$ be a stochastic
process in $\R^d$ adapted to a filtration
$\calF_0, \calF_1, \calF_2, \ldots$. Let $\Delta_n := \xi_{n+1} - \xi_n$.
Suppose that $\liminf_{n \to \infty} n^{-1} \| \xi_n \| \geq c$ a.s., for some constant $c>0$,
  and that 
for some $B < \infty$, $\eps >0$, and $v \in (0,\infty)$ we have
\begin{align}
\label{eq:xi-moms}
 \Exp [ \| \Delta_n \|^2 \mid \calF_n ] & \leq B , \as, \text{ for all } n \in \ZP ; \\
\lim_{n \to \infty} n^\eps \left\| \Exp [ \Delta_n \mid \calF_n ] - v \hat \xi_n \right\| & = 0 , \as 
\label{eq:xi-drift}
\end{align}
Then $\lim_{n\to\infty} n^{-1} \xi_n = v \ell$ a.s.~for some random $\ell\in\Sp{d-1}$.
\end{lemma}

The proof of this result will go by establishing in turn a limiting speed (Lemma~\ref{lem:norm-speed})
and a limiting direction (Lemma~\ref{lem:angles}). First we need a couple of elementary bounds.

\begin{lemma}
For all $x, y \in \R^d$,
\begin{align}
\label{eq:norm-increment}
 \left| \| x +y \| - \| x \| - \hat x \cdot y \right| & \leq \frac{2 \| y \|^2}{\| x \|} .
\end{align}
Moreover, for all $x, y \in \R^d$ with $x \neq 0$ and $x +y \neq 0$,
\begin{align}
\label{eq:angle-increment}
 \left\| \frac{x+y}{\| x +y \|} - \frac{x}{\| x \|} - \frac{y - \hat x ( \hat x  \cdot y)}{\| x \|} \right\| & \leq \frac{3 \| y \|^2}{\| x \| \| x + y \|} .
\end{align}
\end{lemma}
\begin{proof}
First we prove~\eqref{eq:norm-increment}. It suffices to suppose that $x \neq 0$.
We have
\begin{align*}
 \| x + y \| - \| x \| & = \frac{ \| x + y \|^2 - \| x \|^2 }{\| x + y \| + \| x \|}  = \frac{2 x \cdot y + \| y \|^2}{\| x + y \| + \| x \|} .\end{align*}
Hence
\begin{equation}
\label{eq:xy1}
 \left|  \| x + y \| - \| x \| - \frac{2 x \cdot y}{\| x + y \| + \| x \|} \right| \leq \frac{\| y\|^2}{\| x \|} .\end{equation}
Moreover, 
\begin{equation}
\label{eq:xy2}
\left| \frac{1}{\| x + y \| + \| x \|}  -\frac{1}{2 \| x \|} \right| \leq \frac{| \| x+y \| - \| x \| |}{2 \| x \|^2} \leq \frac{\|y\|}{2 \|x\|^2} .\end{equation}
By the triangle inequality, $| \| x + y \| - \| x \| - \hat x \cdot y |$ is bounded above by
\begin{align*}
 \left| \| x + y \| - \| x \| - \frac{2 x \cdot y}{ \| x+y\| + \|x \|} \right| + 2 \| x \cdot y \|
\left| \frac{1}{ \| x+y\| + \|x \|} - \frac{1}{2 \| x \|} \right|, \end{align*}
and then combining~\eqref{eq:xy1} and~\eqref{eq:xy2}, we obtain~\eqref{eq:norm-increment}.

For~\eqref{eq:angle-increment}, suppose that $x \neq 0$ and $x+y \neq 0$. Then, by~\eqref{eq:xy3} and~\eqref{eq:norm-increment},
\begin{align*}
\left\| \frac{ x + y}{\| x + y \|} - \frac{x}{\| x \|} - \frac{y - \hat x ( \hat x \cdot y)}{\| x + y\| } \right\|   = 
 \frac{\left| \| x +y \|- \| x \| - \hat x \cdot y \right|}{\| x + y \|} 
  \leq \frac{2 \| y \|^2}{\|x \| \| x + y \|} .\end{align*}
Now we use the fact that $\| y - \hat x ( \hat x \cdot y) \| \leq \| y\|$ and
\[ \left| \frac{1}{\| x + y\|} - \frac{1}{\|x\|} \right| = \frac{| \| x+y\| - \|x \| |}{ \| x \| \| x + y\|} \leq \frac{\| y\|}{\| x \| \| x +y \|} \]
to get~\eqref{eq:angle-increment}.
\end{proof}

\begin{lemma}
\label{lem:norm-speed}
Let $\zeta_0, \zeta_1, \zeta_2, \ldots$ be a stochastic process on $\RP$ adapted to a filtration
$\calF_0, \calF_1, \calF_2, \ldots$. Suppose that there exist $B < \infty$ and $v \in \R$ such that
\begin{align}
\label{eq:zeta-moms}
\Exp [ ( \zeta_{n+1} - \zeta_n )^2 \mid \calF_n ] & \leq B , \as ; \\
\label{eq:zeta-drift}
 \lim_{n \to \infty}  \left| \Exp [ \zeta_{n+1} - \zeta_n \mid \calF_n ] - v   \right| & = 0 , \as 
\end{align}
Then $\lim_{n\to\infty} n^{-1} \zeta_n = v$, a.s.
\end{lemma}
\begin{proof}
As in the Doob decomposition, let $A_0 :=0$ and $A_n := \sum_{m=0}^{n-1} \Exp [ \zeta_{m+1} - \zeta_m \mid \calF_m]$ for $n \in \N$,
so that
$M_n := \zeta_n - A_n$ is a martingale with $M_0 = \zeta_0$. Moreover,
\begin{align*} \Exp [ M_{n+1}^2 - M_n^2 \mid \calF_n ] & = \Exp [ (M_{n+1} -M_n )^2 \mid \calF_n ]
\\
& \leq \Exp [ ( \zeta_{n+1} - \zeta_n )^2 \mid \calF_n ]  \leq B , \as ,\end{align*}
by~\eqref{eq:zeta-moms}. It follows that,  for any $\eps >0$,
$| M_n | \leq n^{(1/2)+\eps}$ for all but finitely many $n$, a.s.: to see this one may apply 
e.g.~Theorem~2.8.1 of~\cite{bluebook} (take $f(y) = y^2$ and $a(y)=y^{1+\eps}$ in that result).
Hence, a.s.,
$\lim_{n \to \infty} n^{-1} \zeta_n = \lim_{n \to \infty} n^{-1} A_n = v$, 
by~\eqref{eq:zeta-drift}.
\end{proof}

\begin{lemma}
\label{lem:angles}
Let $d \in \N$. Let $\xi_0, \xi_1, \xi_2, \ldots \in \R^d$ be adapted to a filtration
$\calF_0, \calF_1, \calF_2, \ldots$. Let $\Delta_n := \xi_{n+1} - \xi_n$.
Suppose that for some $B < \infty$,~\eqref{eq:xi-moms} holds.
Let $\Delta_n^\perp := \Delta_n - \hat \xi_n (\Delta_n \cdot \hat \xi_n)$.
Suppose also that $\sum_{n=1}^\infty n^{-1} \| \Exp [ \Delta_n^\perp \mid \calF_n ]\| < \infty$ a.s., and,
 for some $c>0$,
 $\liminf_{n \to \infty} n^{-1} \| \xi_n \| \geq c$ a.s.
Then $\lim_{n \to \infty} \hat \xi_n = \ell$ a.s.~for some random $\ell \in \Sp{d-1}$.
\end{lemma}
\begin{proof}
First note that, by Markov's inequality and~\eqref{eq:xi-moms}, for any $q > 0$,  
\begin{align}
\label{eq:angle-big-jump}
 \Exp [ \| \hat \xi_{n+1} - \hat \xi_n \|^q  \1 \{ \| \Delta_n \| > \tfrac{1}{2} \| \xi_n \| \} \mid \calF_n ]
& \leq 2^q \Pr ( \| \Delta_n \|^2 > \tfrac{1}{4} \| \xi_n \|^2 \mid \calF_n  ) \nonumber\\
& \leq   2^{2+q}  B \| \xi_n\|^{-2} , \as \end{align}
On the other hand, we apply~\eqref{eq:angle-increment} with $x = \xi_n$ and $y = \Delta_n$ to get, on $\{ \xi_n \neq 0\}$,
\begin{align*}
\left\| \hat \xi_{n+1} - \hat \xi_n - \frac{\Delta_n^\perp}{\| \xi_n \|} \right\| \1 \{ \| \Delta_n \| \leq \tfrac{1}{2} \| \xi_n \| \}
\leq \frac{6 \| \Delta_n \|^2}{\| \xi_n \|^2} .\end{align*}
Noting that
$\Exp [ \| \Delta_n^\perp \| \1 \{ \| \Delta_n \| > \tfrac{1}{2} \| \xi_n \| \} \mid \calF_n ]
\leq 2 \| \xi_n \|^{-1} \Exp [ \| \Delta_n \|^2 \mid \calF_n]$, by~\eqref{eq:xi-moms} we get
\begin{align}
\label{eq:angle-drift} 
\left\|
\Exp [  \hat \xi_{n+1} - \hat \xi_n 
\mid \calF_n ] - \frac{ \Exp [ \Delta_n^\perp \mid \calF_n ]}{ \| \xi_n \| } \right\| \leq 16 B \| \xi_n\|^{-2} , \text{ on } \{ \xi_n \neq 0\},
\end{align}
using the $q=1$ case of~\eqref{eq:angle-big-jump}.
It follows from~\eqref{eq:xy3} that
$\| \hat \xi_{n+1} - \hat \xi_n \| \leq 2 \frac{\| \Delta_n \|}{\| \xi_n + \Delta_n \|}$, so 
\[ \Exp [ \| \hat \xi_{n+1} - \hat \xi_n \|^2  \1 \{ \| \Delta_n \| \leq \tfrac{1}{2} \| \xi_n \| \} \mid \calF_n ]
\leq 16 B \| \xi_n\|^{-2}, \as \]
Together with the  $q=2$ case of~\eqref{eq:angle-big-jump}, this implies
\[  \Exp [ \| \hat \xi_{n+1} - \hat \xi_n \|^2    \mid \calF_n ] \leq 32 B \| \xi_n\|^{-2}, \as \]
Define $A_n := \sum_{m=0}^{n-1} \Exp [ \hat \xi_{m+1} - \hat \xi_m \mid \calF_m ]$,
so that $M_n := \hat \xi_n - A_n$ is a martingale in $\R^d$.
Now
\[ \Exp [ \| M_{n+1} - M_n \|^2 \mid \calF_n ] \leq \Exp [ \| \hat \xi_{n+1} - \hat \xi_n \|^2 \mid \calF_n ] \leq 32 B \| \xi_n \|^{-2} .\]
But $\| \xi_n \| > (c/2) n$ for all but finitely many $n$, so we get
$\sum_{n =0}^\infty \Exp [ \| M_{n+1} - M_n \|^2 \mid \calF_n ] < \infty$, a.s.
It follows that $M_n \to M_\infty$ a.s.~for some $M_\infty \in\R^d$, by e.g.~the $d$-dimensional
version of Theorem~5.4.9 of~\cite{durrett}.
Hence  for $\hat \xi_n$ to converge a.s.,
it is sufficient that $\lim_{n \to \infty} A_n$ exists a.s., and, by~\eqref{eq:angle-drift},
sufficient for this is that 
$\sum_{n=1}^\infty n^{-1} \| \Exp [ \Delta_n^\perp \mid \calF_n ]\| < \infty$~a.s.
Also, since $\| \xi_n \| \to \infty$, the limit of $\hat \xi_n$ is non-zero.
\end{proof}

Now we can complete the proof of Lemma~\ref{lem:ballistic}.

\begin{proof}[Proof of Lemma~\ref{lem:ballistic}]
Taking $x = \xi_n$ and $y=\Delta_n$ in~\eqref{eq:norm-increment}, taking conditional expectations, and using~\eqref{eq:xi-moms}, we obtain
\[ \left| \Exp [ \| \xi_{n+1} \| - \| \xi_n \| \mid \calF_n ] - \hat \xi_n \cdot \Exp [ \Delta_n \mid \calF_n ] \right| \leq 
2 B \| \xi_n \|^{-1} .\]
Then by assumption~\eqref{eq:xi-drift} and the fact that $\| \xi_n \| > (c/2)n$ for all but finitely many $n$,
\[ \lim_{n\to\infty} n^\eps \left| \Exp [ \| \xi_{n+1} \| - \| \xi_n \| \mid \calF_n ] - v \right| = 0, \as \]
So we can apply Lemma~\ref{lem:norm-speed} with $\zeta_n = \| \xi_n\|$ to deduce that
$\lim_{n\to\infty} n^{-1} \| \xi_n \| =v$, a.s. Moreover,
it also follows from~\eqref{eq:xi-drift}
that
$n^\eps \| \Exp [ \Delta_n^\perp \mid \calF_n ] \| \to 0$, a.s. Hence the conditions of Lemma~\ref{lem:angles}
are also satisfied, and we conclude that $\lim_{n\to\infty}  \hat \xi_n  = \ell$ a.s., for some $\ell \in \Sp{d-1}$.
Then $\lim_{n\to\infty} n^{-1} \xi_n = \lim_{n \to \infty}  n^{-1} \| \xi_n \| \hat \xi_n = v \ell$, a.s.
\end{proof}

The next result, which shows how local speed translates to global speed, is an important
ingredient in the proof of Theorem~\ref{thm:memory-one}.

\begin{lemma}
\label{lem:norms}
Let $d \in \N$. Let $\xi_0, \xi_1, \xi_2, \ldots$ be a stochastic
process in $\R^d$ with  $\xi_0 = 0$, such that, for some constant $B < \infty$,
\[ \Pr ( \| \xi_{n+1} - \xi_n \| \leq B ) = 1, \text{ for all } n \in \ZP ,\]
and suppose that $\| \xi_n \| \to \infty$, a.s. Then
\[ \lim_{n \to \infty} \left| \frac{1}{n} \Exp \| \xi_n \|
- \frac{1}{n} \sum_{m=0}^{n-1} \Exp [ ( \xi_{m+1} - \xi_m) \cdot \hat \xi_m ] \right| = 0 .\]
In particular, 
if $\lim_{n \to \infty} \Exp [ ( \xi_{n+1} - \xi_n) \cdot \hat \xi_n ] = v \in [0,\infty]$
then $n^{-1} \Exp \| \xi_n\| \to v$ as well.
\end{lemma}
\begin{proof}
Let $\Delta_n := \xi_{n+1} - \xi_n$.
We have from~\eqref{eq:norm-increment} that for any $y$ with $\| y \| \leq B$,
\begin{equation}
\label{eq1}
  \left| \| x + y \| - \| x \| - \hat x \cdot y \right| \leq C (1 + \| x\|)^{-1} .\end{equation}
It follows from an application of~\eqref{eq1}  with $x = \xi_m$ and $y = \Delta_m$ that
\[ \Exp \| \xi_n \| = \sum_{m=0}^{n-1} \Exp [ \| \xi_{m} +\Delta_m \| - \| \xi_m \| ]
=  \sum_{m=0}^{n-1} \Exp [ \hat \xi_m \cdot \Delta_m ] + \sum_{m=0}^{n-1} \Exp \zeta_m ,\]
where $| \zeta_m | \leq C (1 + \| \xi_m \|)^{-1}$. Since $\| \xi_m \| \to \infty$ a.s.,
the bounded convergence theorem implies that $\Exp \zeta_m \to 0$, and the claimed result follows.
\end{proof}

{ 
\section*{Acknowledgements}
F.C.~and M.M.~acknowledge the support of the project SWIWS (ANR-17-CE40-0032). The authors
are grateful to two anonymous referees for their comments and suggestions.

}
\end{document}